\documentclass[12pt]{amsart}

\usepackage{amssymb}
\usepackage[margin=1.52in]{geometry}

\usepackage{amsthm,amsmath}
\usepackage{graphicx}
\usepackage{amscd}
\usepackage{latexsym}
\usepackage{mathrsfs}
\usepackage{xcolor}
\usepackage{enumerate}

\usepackage{soul}

\usepackage{tikz}
\usepackage{tikz}
\usetikzlibrary{arrows,snakes,backgrounds,patterns}

\makeatletter
\def\moverlay{\mathpalette\mov@rlay}
\def\mov@rlay#1#2{\leavevmode\vtop{%
\baselineskip\z@skip \lineskiplimit-\maxdimen
\ialign{\hfil$\m@th#1##$\hfil\cr#2\crcr}}}
\newcommand{\charfusion}[3][\mathord]{
#1{\ifx#1\mathop\vphantom{#2}\fi
     \mathpalette\mov@rlay{#2\cr#3}}
     \ifx#1\mathop\expandafter\displaylimits\fi}
\makeatother
\newcommand{\cupdot}{\charfusion[\mathbin]{\cup}{\cdot}}

\newcommand\abs[1]{\lvert#1\rvert}

\newcommand{\reals}{{\mathbb R}}

\newcommand{\DD}{\Delta}
\newcommand{\C}{\mathbf{c}}
\newcommand{\D}{\mathbf{d}}
\newcommand{\conv}{\mathrm{conv}}

\theoremstyle{plain}
\newtheorem{theorem}{Theorem}
\newtheorem{lemma}[theorem]{Lemma}
\newtheorem{corollary}[theorem]{Corollary}
\newtheorem{proposition}[theorem]{Proposition}

\theoremstyle{definition}
\newtheorem{definition}[theorem]{Definition}
\newtheorem{example}[theorem]{Example}

\theoremstyle{remark}

\begin{document}

\title{Facial structures of lattice path matroid polytopes}

\author{Suhyung An}
\address{Department of Mathematics\\
         Yonsei University\\
         Seoul 120-749, Republic of Korea}
\email{hera1973@yonsei.ac.kr}

\author{JiYoon Jung}
\address{Department of Mathematics\\
         Marshall University\\
         Huntington, WV 25755}
\email{jungj@marshall.edu}

\author{Sangwook Kim}
\address{Department of Mathematics\\
         Chonnam National University\\
         Gwangju 500-757, Republic of Korea}
\email{swkim.math@chonnam.ac.kr}

\begin{abstract}
   A lattice path matroid is a transversal matroid corresponding to 
   a pair of lattice paths on the plane.
   A matroid base polytope is the polytope whose vertices are the incidence vectors
   of the bases of the given matroid.
   In this paper, we study facial structures of matroid base polytopes corresponding to
   lattice path matroids.
\end{abstract}


\maketitle


\section{Introduction}
\label{sec-introduction}

   For a matroid $M$ on a ground set $[n]:= \{1, 2, \dots, n \}$ 
   with a set of bases $\mathcal{B}(M)$, 
   a \emph{matroid base polytope} $\mathcal{P}(M) := \mathcal{P}(\mathcal{B}(M))$ of $M$ is
   the polytope in $\reals^n$ whose vertices are the incidence vectors of the bases of $M$.
   The polytope $\mathcal{P}(M)$ is a face of a \emph{matroid independence polytope} 
   first studied by Edmonds~\cite{Edmonds},
   whose vertices are the incidence vectors of all the independent sets in $M$.
   There are a lot of research activities about matroid base polytopes for the last few years
   since it has various applications in 
   algebraic geometry, combinatorial optimization, Coxeter group theory, and tropical geometry.
   In general, matroid base polytopes are not well understood.
   
   A \emph{lattice path matroid} is a transversal matroid 
   corresponding to a pair of lattice paths having common end points.
   Many interesting and striking properties of these matroids have been studied.
   The combinatorial and structural properties of lattice path matroids are given 
   by Bonin et al. in \cite{BoninMierNoy} and \cite{BoninMier}.
   The $h$-vectors, Bergman complexes, and Tutte polynomials of lattice path matroids are studied
   by several authors~\cite{DelucchiDlugosch, MortonTurner, Schweig2010}.

   In this paper, we study the facial structure of a \emph{lattice path matroid polytope} 
   which is a matroid base polytope corresponding to a lattice path matroid.
   This class of matroid base polytopes is belong to important classes of polytopes 
   such as positroid polytopes and generalized permutohedra.
   Positroid polytopes are studied by Ardila et al.~\cite{ArdilaRinconWilliams}
   and generalized permutohedra are studied 
   by Postnikov and other authors~\cite{Oh, PostnikovReinerWilliams, Postnikov}.
   Bidkhori~\cite{Bidkhori} provides a description of facets of a lattice path matroid polytope
   and we extend it to all the faces.
   
   This paper is organized as follows.
   In Section~\ref{sec-lattice-path-matroids}, 
   definitions and properties of lattice path matroids are given.
   In Section~\ref{sec-matroid-base-polytopes}, 
   we define lattice path matroid polytopes and give known results about them.
   In Section~\ref{sec-border strips}, 
   lattice path matroid polytopes for the case of border strips are studied.
   In particular, 
   we show that all the faces of a lattice path matroid polytope in this case can be described by
   certain subsets of deletions, contractions and direct sums and 
   express them in terms of a lattice path obtained from the border strip.
   Section~\ref{sec-skew-shape} explains 
   facial structures of a lattice path matroid polytope for a general case 
   in terms of certain tilings of skew shapes inside the given region.

\section{Lattice path matroids}
\label{sec-lattice-path-matroids}

   In this section, we provide basic definitions and properties of lattice path matroids.

   A \emph{matroid $M$} is a pair $(E(M), \mathcal{B}(M))$ consisting of a finite set $E(M)$ and
   a collection $\mathcal{B}(M)$ of subsets of $E(M)$ that satisfy the following conditions:
   \begin{enumerate}
      \item
      $\mathcal{B}(M) \ne \emptyset$, and
      \item
      for each pair of distinct sets $B, B'$ in $\mathcal{B}(M)$ and 
      for each element $x \in B - B'$,
      there is an element $y \in B' - B$ such that 
      $(B - \{ x \}) \cup \{ y \} \in \mathcal{B}(M)$.
   \end{enumerate}

   The set $E(M)$ is called the \emph{ground set} of $M$ and 
   the sets in $\mathcal{B}(M)$ are called the \emph{bases} of $M$.
   Subsets of bases are called the \emph{independent sets} of $M$.
   It is easy to see that 
   all the bases of $M$ have the same cardinality, called the \emph{rank} of $M$.

   A set system $\mathcal{A} = \{ A_j : j \in J \}$ is a multiset of subsets of a finite set~$S$.
   A \emph{transversal} of $\mathcal{A}$ is 
   a set $\{ x_j : x_j \in A_j \text{ for all } j \in J \}$ of $\abs{J}$ distinct elements.
   A \emph{partial transversal} of $\mathcal{A}$ is 
   a transversal of a set system of the form $\{A_k : k \in K$ with $K \subseteq J \}$.

   Edmonds and Fulkerson \cite{EdmondsFulkerson} show the following fundamental result:

   \begin{theorem}
      The transversals of a set system $\mathcal{A} = \{ A_j : j \in J \}$ form 
      the bases of a matroid on $S$.
   \end{theorem}

   A \emph{transversal matroid} is a matroid whose bases are the transversals of some set system
   $\mathcal{A} = \{ A_j : j \in J \}$.
   The set system $\mathcal{A}$ is called the \emph{presentation} of the transversal matroid.
   The independent sets of a transversal matroid are the partial transversals of $\mathcal{A}$.

   We consider lattice paths in the plane using steps $E = (1,0)$ and $N = (0,1)$.
   The letters are abbreviations of East and North.
   We will often treat lattice paths as words in the alphabets $E$ and $N$, and 
   we will use the notation~$\alpha^n$
   to denote the concatenation of $n$ letters of $\alpha$.
   The \emph{length} of a lattice path $P = p_1 p_2 \cdots p_n$ is $n$, 
   the number of steps in $P$.

   \begin{definition}
      Let $P = p_1 p_2 \cdots p_{m+r}$ and $Q = q_1 q_2 \cdots q_{m+r}$ be two lattice paths 
      from $(0,0)$ to $(m, r)$ with $P$ never going above $Q$.
      Let $\{ p_{u_1}, p_{u_2}, \dots, p_{u_r} \}$ be the set of North steps of $P$, 
      with $u_1 < u_2 < \cdots < u_r$;
      similarly, let $\{ q_{l_1}, q_{l_2}, \dots, q_{l_r} \}$ be the set of North steps of $Q$
      with $l_1 < l_2 < \cdots < l_r$.
      Let $N_i$ be the interval $[l_i, u_i]$ of integers.
      Let $M(P,Q)$ be the transversal matroid 
      that has ground set $[m+r]$ and presentation $\{ N_i: i \in [r] \}$.
      The pair $(P,Q)$ is a \emph{lattice path presentation} of $M(P,Q)$.
      A \emph{lattice path matroid} is a matroid~$M$ 
      that is isomorphic to $M(P,Q)$ for some such pair of lattice paths $P$ and $Q$.
      We will sometimes call a lattice path presentation of $M$ simply a presentation of $M$
      when there is no danger of confusion and when doing so avoids awkward repetition.
    \end{definition}
    
    The fundamental connection between the transversal matroid $M(P,Q)$ and 
    the lattice paths that stay in the region bounded by $P$ and $Q$ is 
    the following theorem of Bonin et al.~\cite{BoninMierNoy}.

    \begin{theorem}
    \label{thm-bases-of-lattice-path-matroids}
       A subset $B$ of $[m+r]$ with $\abs{B} = r$ is a basis of $M(P,Q)$ 
       if and only if 
       the associated lattice path $P(B)$ stays in the region bounded by $P$ and $Q$, 
       where $P(B)$ is the path which has its north steps on the set $B$
       positions and east steps elsewhere.
    \end{theorem}

\section{Lattice path matroid polytopes}
\label{sec-matroid-base-polytopes}

   Let $\mathcal{B}$ be a collection of $r$-element subsets of $[n]$.
   For each subset $B = \{ b_1, \dots, b_r \}$ of $[n]$, let
   $$
   e_B = e_{b_1} + \cdots + e_{b_r} \in \reals^n ,
   $$
   where $e_i$ is the $i$th standard basis vector of $\reals^n$.
   The collection $\mathcal{B}$ is represented by the convex hull of these points
   $$
   \mathcal{P}(\mathcal{B}) = \conv \{ e_B : B \in \mathcal{B} \}.
   $$
   This is a convex polytope of dimension $\le n-1$
   and is a subset of the $(n-1)$-simplex
   $$
   \DD_n = \{ (x_1, \dots, x_n) \in \reals^n : x_1 \ge 0,
         \dots, x_n \ge 0, x_1 + \cdots + x_n = r \}.
   $$
   Gel$'$fand, Goresky, MacPherson, and Serganova~\cite[Thm. 4.1]{GelfandGoreskyMacPhersonSerganova}
   show the following characterization of matroid base polytopes.
   \begin{theorem}
   \label{thm-matroid-polytopes}
      The subset $\mathcal{B}$ is the collection of bases of a matroid if and only if
      every edge of the polytope $\mathcal{P}(\mathcal{B})$ is parallel to
      a difference $e_\alpha - e_\beta$ of two distinct standard basis vectors.
   \end{theorem}

   By the definition, the vertices of $\mathcal{P}(M)$ represent the bases of $M$.
   For two bases $B$ and $B'$ in $\mathcal{B}(M)$,
   $e_B$ and $e_{B'}$ are connected by an edge if and only if
   $e_B - e_{B'} = e_\alpha - e_\beta$.
   Since the latter condition is equivalent to $B - B' = \{ \alpha \}$
   and $B' - B = \{ \beta \}$, the edges of $\mathcal{P}(M)$ represent
   the basis exchange axiom.
   The basis exchange axiom gives the following equivalence relation
   on the ground set $[n]$ of the matroid $M$:
   $\alpha$ and $\beta$ are \emph{equivalent} 
   if there exist bases $B$ and $B'$ in $\mathcal{B}(M)$
   with $\alpha \in B$ and $B' = (B - \{ \alpha \} ) \cup \{ \beta \}$.
   The equivalence classes are called the \emph{connected components} of $M$.
   The matroid $M$ is called \emph{connected} if it has only one connected component.
   Feichtner and Sturmfels~\cite[Prop. 2.4]{FeichtnerSturmfels} express 
   the dimension of the matroid base polytope $\mathcal{P}(M)$ in terms of 
   the number of connected components of $M$.

   \begin{proposition}
   \label{dim-of-matroid-polytope}
      Let $M$ be a matroid on $[n]$.
      The dimension of the matroid base polytope $\mathcal{P}(M)$ equals $n - c(M)$, where
      $c(M)$ is the number of connected components of $M$.
   \end{proposition}

   Bonin et al.~\cite{BoninMierNoy} give the following result explaining 
   the number of connected components of the lattice path matroid.

   \begin{proposition}
      The lattice path matroid is connected if and only if 
      the bounding lattice paths $P$ and $Q$ meet only at $(0,0)$ and $(m,r)$.
   \end{proposition}

   Remind that, 
   for a skew shape bounded by lattice paths $P$ and $Q$ (denoted by $[P,Q]$) and
   a related rank $r$ lattice path matroid $M[P,Q]$ with a set of bases $\mathcal{B}(M[P,Q])$, 
   the lattice path matroid polytope $\mathcal{P}(M[P,Q])$ is the convex hull 
   $\mathcal{P}(\mathcal{B}(M[P,Q])) 
   = \conv \{ e_B  = e_{b_1} + \cdots + e_{b_r} : B = \{ b_1, \dots, b_r \} \in~\mathcal{B} \}$ 
   where $e_i$ is the $i$th standard basis vector of $\reals^{m+r}$.
   The following result about the dimension of the lattice path matroid polytope 
   immediately follows from these results.

   \begin{corollary}
   \label{dimension-of-matroid-polytope}
      For lattice paths $P$ and $Q$ from $(0,0)$ to $(m, r)$ with $P$ never going above~$Q$, 
   the dimension of the lattice path matroid polytope $\mathcal{P}(M[P,Q])$ is $m+r-k+1$, 
   where $k$ is the number of intersection points of $P$ and $Q$.
   \end{corollary}

\section{Lattice path matroid polytopes for border strips}
\label{sec-border strips}

   For a matroid $(E(M), \mathcal{B}(M))$ and a subset $S$ of $E(M)$, 
   let $r(S)$ denote the \emph{rank} of $S$, the size of largest independent subsets of $S$.    
   Then, the \emph{restriction} $M|_S$ is the matroid on $S$ having the bases 
   $\mathcal{B} (M|_S) = \{ B \cap S : B \in \mathcal{B}(M) \text{ and } |B \cap S| = r(S)\}$, 
   and 
   the \emph{contraction} $M/S$ is the matroid on $E(M) - S$ having the bases 
   $\mathcal{B} (M/S) = \{ B - S : B \in \mathcal{B}(M) \text{ and } |B \cap~S| = r(S) \}$. 
   For two matroids $(E(M_1), \mathcal{B}(M_1))$ and $(E(M_2), \mathcal{B}(M_2))$, 
   their \emph{direct sum} $M_1 \oplus M_2$ is 
   the matroid on $E(M_1) \cupdot E(M_2)$ having the bases 
   $\mathcal{B}(M_1 \oplus~M_2)
   = \{B_1 \cup B_2 : B_1 \in \mathcal{B}(M_1), B_2 \in \mathcal{B}(M_2)\}$.

   \begin{theorem}[Bonin and de Mier~\cite{BoninMier}, 2006]
   \label{class-of-lattice-path-matroid}
      The class of lattice path matroids is closed 
      under restrictions, contractions, and direct sums.
   \end{theorem}
   
   The connection between constructions of lattice path matroids and 
   skew shapes bounded by lattice paths $P$ and $Q$ is known as follows. 
   First, we label a step of skew shape $[P,Q]$ by $i+j+1$ if the step begins at $(i, j)$. 
   Then, a restriction and a contraction of lattice path matroid $M[P,Q]$ correspond to 
   the deletion of the corresponding region in skew shape $[P,Q]$.
   If the begin point of one skew shape is attached to the end point of the other, and 
   all the steps are relabeled, we get a direct sum of two matroids.          
   Figure~\ref{R-C-D-LPM} shows 
   how a restriction, a contraction, and a direct sum work to $[P,Q]$.

   \begin{figure}[h]
      \begin{center}
         \begin{tikzpicture}[scale=.45]
         
                   \draw (0,0)--(1,0)--(1,1)--(0,1)--cycle;
                   \draw (1,0)--(2,0)--(2,1)--(1,1)--cycle;
                   \draw (1,1)--(2,1)--(2,2)--(1,2)--cycle;
                   \draw (0.3,0.5) node{\scriptsize{$1$}};
                   \draw (1.3,0.5) node{\scriptsize{$2$}};
                   \draw (2.3,0.5) node{\scriptsize{$3$}};
                   \draw (1.3,1.5) node{\scriptsize{$3$}};
                   \draw (2.3,1.5) node{\scriptsize{$4$}};
                   
                   \draw (1,-1.3) node{\footnotesize{[P,Q]}};
                   
                   \filldraw[orange, fill=orange]                    (5,1)--(5,1.3)--(6,1.3)--(6,1)--(7,0)--(7,-0.3)--(6,-0.3)--(6,0)--cycle;
                   \draw (5,0)--(6,0)--(6,1)--(5,1)--cycle;
                   \draw (6,0)--(7,0)--(7,1)--(6,1)--cycle;
                   \draw (6,1)--(7,1)--(7,2)--(6,2)--cycle;
                                      
                   \draw[ultra thick,orange] (9,1)--(10,0);
                   \draw (9,0)--(10,0)--(10,1)--(9,1)--cycle;
                   \draw (9,1)--(10,1)--(10,2)--(9,2)--cycle;
                   \draw (9.3,0.5) node{\scriptsize{$1$}};
                   \draw (9.3,1.5) node{\scriptsize{$3$}};
                   \draw (10.3,0.5) node{\scriptsize{$3$}};
                   \draw (10.3,1.5) node{\scriptsize{$4$}};
                                      
                   \draw (8,1) node{$\rightarrow$};
                   \draw (8,-1.3) node{\footnotesize{Restriction}};
                                      
                   \filldraw[orange, fill=orange] (14,1)--(13.7,1)--(13.7,2)--(14,2)--(15,1)--(15.3,1)--(15.3,0)--(15,0)--cycle;
                   \draw (13,0)--(14,0)--(14,1)--(13,1)--cycle;
                   \draw (14,0)--(15,0)--(15,1)--(14,1)--cycle;
                   \draw (14,1)--(15,1)--(15,2)--(14,2)--cycle;
                   
                   \draw[ultra thick,orange] (18,1)--(19,0);
                   \draw (17,0)--(18,0)--(18,1)--(17,1)--cycle;
                   \draw (18,0)--(19,0)--(19,1)--(18,1)--cycle;
                   \draw (17.3,0.5) node{\scriptsize{$1$}};
                   \draw (18.3,0.5) node{\scriptsize{$2$}};
                   \draw (19.3,0.5) node{\scriptsize{$4$}};
                   
                   \draw (16,1) node{$\rightarrow$};
                   \draw (16,-1.3) node{\footnotesize{Contraction}};                                      
                   
                   \draw[orange, fill=orange] (24,2) circle (1ex);
                   \draw (22,0)--(23,0)--(23,1)--(22,1)--cycle;
                   \draw (23,0)--(24,0)--(24,1)--(23,1)--cycle;
                   \draw (23,1)--(24,1)--(24,2)--(23,2)--cycle;
                   
                   \draw[orange, fill=orange] (24,2.5) circle (1ex);
                   \draw (24,2.5)--(25,2.5)--(25,3.5)--(24,3.5)--cycle;
                   \draw (24.3,3) node{\scriptsize{$1$}};
                   \draw (25.3,3) node{\scriptsize{$2$}};
                   
                   \draw[orange, fill=orange] (28,2) circle (1ex);  
                   \draw (26,0)--(27,0)--(27,1)--(26,1)--cycle;
                   \draw (27,0)--(28,0)--(28,1)--(27,1)--cycle;
                   \draw (27,1)--(28,1)--(28,2)--(27,2)--cycle;
                   \draw (26.3,0.5) node{\scriptsize{$1$}};
                   \draw (27.3,0.5) node{\scriptsize{$2$}};
                   \draw (28.3,0.5) node{\scriptsize{$3$}};
                   \draw (27.3,1.5) node{\scriptsize{$3$}};
                   \draw (28.3,1.5) node{\scriptsize{$4$}};                      
                   
                   \draw (28,2)--(29,2)--(29,3)--(28,3)--cycle;        
                   \draw (28.3,2.5) node{\scriptsize{$5$}};
                   \draw (29.3,2.5) node{\scriptsize{$6$}};
                   
                   \draw (25,1) node{$\rightarrow$};
                   \draw (25,-1.3) node{\footnotesize{Direct Sum}};
                   
         \end{tikzpicture}
      \end{center}
   \vspace{-3mm}   
   \caption{Restriction, contraction, and direct sum of matroids on $[P,Q]$.}   
   \label{R-C-D-LPM}
   \end{figure}
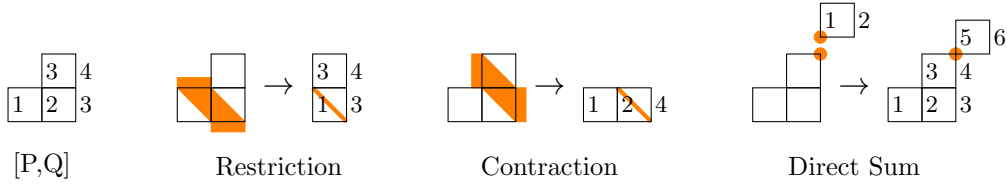  
   
   Our main question is about lattice path matroid polytopes 
   regarding to restrictions, contractions, and direct sums of matroids. 
   Can we use constructions of matroids 
   to figure out the properties of lattice path matroid polytopes? 
   To study facets and describe facial structures of lattice path matroid polytopes, 
   we introduce more specific definitions and notations. 
    
   The \emph{$i$-deletion} of $M[P,Q]$ is defined by a matroid $M|_{E(M)-\{i\}}$, 
   the restriction of $M[P,Q]$ on $E(M)-\{i\}$. 
   The \emph{$i$-contraction} of $M[P,Q]$ is a matroid $M/{\{i\}} \oplus \{i\}$ 
   which is isomorphic to $M/{\{i\}}$, the contraction of $M[P,Q]$ on~\{i\}. 
   An \emph{outside corner} $(p,q)$ of the region $[P,Q]$ is 
   a point on a corner $NE$ on~$P$ or a corner $EN$ on $Q$. 
   At the outside corner $(p,q)$, 
   let \emph{$(p,q)$-direct sum} of $M[P,Q]$ be a matroid $M_1 \oplus M_2$ 
   where a matroid $M_1$ and $M_2$ correspond to 
   the lower left quadrant and the upper right quadrant of $[P,Q]$ 
   with the center $(p,q)$, respectively. 
   If $(p,q)$ is the unique outside corner of $[P,Q]$ such that $p+q = i$ for some integer $i$, 
   then $(p,q)$-direct sum is abbreviated to \emph{$i$-direct sum}. 
   Figure~\ref{i-R-C-D-LPM} shows 
   how a deletion, a contraction, and a direct sum work to $[P,Q]$. 
    
   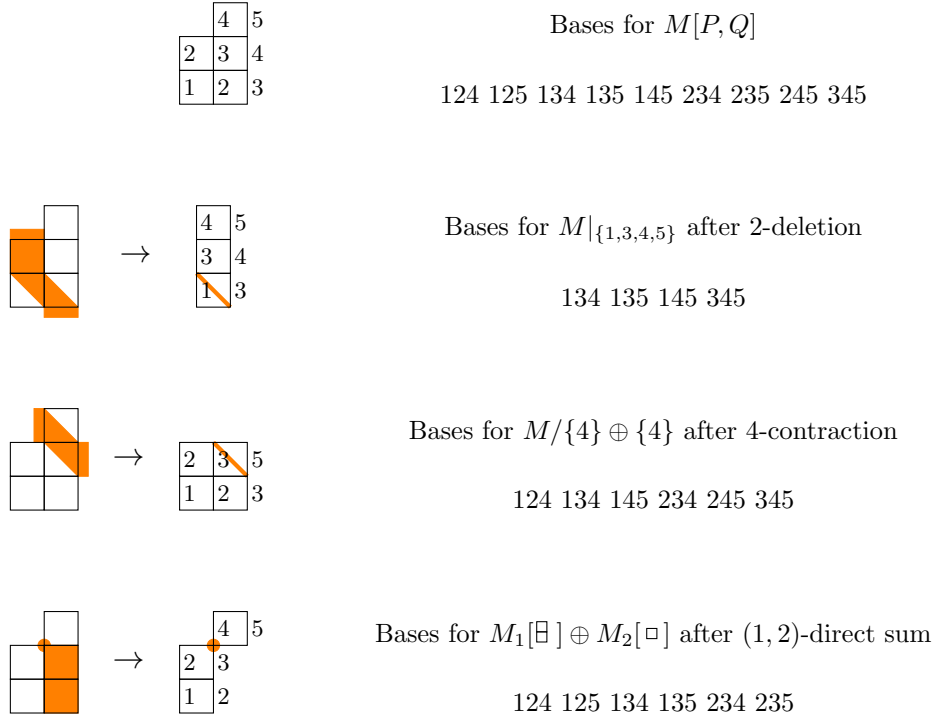
\begin{figure}[h]
      \begin{center}
         \begin{tikzpicture}[scale=.45]
         
                   \draw (0,18)--(1,18)--(1,19)--(0,19)--cycle;
                   \draw (1,18)--(2,18)--(2,19)--(1,19)--cycle;
                   \draw (1,19)--(2,19)--(2,20)--(1,20)--cycle;
		   \draw (0,19)--(1,19)--(1,20)--(0,20)--cycle;
                   \draw (1,20)--(2,20)--(2,21)--(1,21)--cycle;
                   \draw (0.3,18.5) node{\scriptsize{$1$}};
                   \draw (1.3,18.5) node{\scriptsize{$2$}};
                   \draw (2.3,18.5) node{\scriptsize{$3$}};
                   \draw (0.3,19.5) node{\scriptsize{$2$}};
                   \draw (1.3,19.5) node{\scriptsize{$3$}};
                   \draw (2.3,19.5) node{\scriptsize{$4$}};  
                   \draw (1.3,20.5) node{\scriptsize{$4$}};
                   \draw (2.3,20.5) node{\scriptsize{$5$}};    
                                                         
                   \filldraw[orange, fill=orange] (-5,13)--(-5,14.3)--(-4,14.3)--(-4,14)--(-4,13)--(-3,12)--(-3,11.7)--(-4,11.7)--(-4,12)--cycle;                  
                   \draw (-5,12)--(-4,12)--(-4,13)--(-5,13)--cycle;
                   \draw (-4,12)--(-3,12)--(-3,13)--(-4,13)--cycle;
                   \draw (-4,13)--(-3,13)--(-3,14)--(-4,14)--cycle;
		   \draw (-5,13)--(-4,13)--(-4,14)--(-5,14)--cycle;
                   \draw (-4,14)--(-3,14)--(-3,15)--(-4,15)--cycle;
                   \draw (-1.3,13.5) node{$\rightarrow$};   
                   
                   \draw[ultra thick,orange] (0.5,13)--(1.5,12);
                   \draw (0.5,12)--(1.5,12)--(1.5,13)--(0.5,13)--cycle;
                   \draw (0.5,13)--(1.5,13)--(1.5,14)--(0.5,14)--cycle;
                   \draw (0.5,14)--(1.5,14)--(1.5,15)--(0.5,15)--cycle;
                   \draw (0.8,12.5) node{\scriptsize{$1$}};
                   \draw (1.8,12.5) node{\scriptsize{$3$}};
                   \draw (0.8,13.5) node{\scriptsize{$3$}};
                   \draw (1.8,13.5) node{\scriptsize{$4$}};
                   \draw (0.8,14.5) node{\scriptsize{$4$}};
                   \draw (1.8,14.5) node{\scriptsize{$5$}};                   
                                      
                   \filldraw[orange, fill=orange] (-4,8)--(-4.3,8)--(-4.3,9)--(-4,9)--(-3,8)--(-2.7,8)--(-2.7,7)--(-3,7)--cycle;
                   \draw (-5,6)--(-4,6)--(-4,7)--(-5,7)--cycle;
                   \draw (-4,6)--(-3,6)--(-3,7)--(-4,7)--cycle;
                   \draw (-4,7)--(-3,7)--(-3,8)--(-4,8)--cycle;
		   \draw (-5,7)--(-4,7)--(-4,8)--(-5,8)--cycle;
                   \draw (-4,8)--(-3,8)--(-3,9)--(-4,9)--cycle;                   
                   \draw (-1.5,7.5) node{$\rightarrow$};
                   
                   \draw[ultra thick,orange] (1,8)--(2,7);
                   \draw (0,6)--(1,6)--(1,7)--(0,7)--cycle;
                   \draw (1,6)--(2,6)--(2,7)--(1,7)--cycle;
                   \draw (1,7)--(2,7)--(2,8)--(1,8)--cycle;
		   \draw (0,7)--(1,7)--(1,8)--(0,8)--cycle;
                   \draw (0.3,6.5) node{\scriptsize{$1$}};
                   \draw (1.3,6.5) node{\scriptsize{$2$}};
                   \draw (2.3,6.5) node{\scriptsize{$3$}};
                   \draw (0.3,7.5) node{\scriptsize{$2$}};
                   \draw (1.3,7.5) node{\scriptsize{$3$}};
                   \draw (2.3,7.5) node{\scriptsize{$5$}};  
                                      
                   \filldraw[orange, fill=orange] (-4,0) rectangle (-3,2);
                   \draw[orange, fill=orange] (-4,2) circle (-1ex);  
                   \draw (-5,0)--(-4,0)--(-4,1)--(-5,1)--cycle;
                   \draw (-4,0)--(-3,0)--(-3,1)--(-4,1)--cycle;
                   \draw (-4,1)--(-3,1)--(-3,2)--(-4,2)--cycle;
		   \draw (-5,1)--(-4,1)--(-4,2)--(-5,2)--cycle;
                   \draw (-4,2)--(-3,2)--(-3,3)--(-4,3)--cycle;                 
                   \draw (-1.5,1.5) node{$\rightarrow$};                                 
                   
                   \draw[orange, fill=orange] (1,2) circle (1ex);
                   \draw (0,0)--(1,0)--(1,1)--(0,1)--cycle;
                   \draw (0,1)--(1,1)--(1,2)--(0,2)--cycle;
                   \draw (1,2)--(2,2)--(2,3)--(1,3)--cycle;
                   \draw (0.3,0.5) node{\scriptsize{$1$}};
                   \draw (0.3,1.5) node{\scriptsize{$2$}};
                   \draw (1.3,0.5) node{\scriptsize{$2$}};
                   \draw (1.3,1.5) node{\scriptsize{$3$}};
                   \draw (1.3,2.5) node{\scriptsize{$4$}};
                   \draw (2.3,2.5) node{\scriptsize{$5$}};
                                       
                   \draw (14,18.3) node{\footnotesize{$124~125~134~135~145~234~235~245~345$}};
                   \draw (14,20.3) node{\footnotesize{Bases for $M[P,Q]$}};
                   
                   \draw (14,12.3) node{\footnotesize{$134~135~145~345$}};
                   \draw (14,14.3) node{\footnotesize{Bases for $M|_{\{1,3,4,5\}}$ after $2$-deletion}};
                
                   \draw (14,6.3) node{\footnotesize{$124~134~145~234~245~345$}};
                   \draw (14,8.3) node{\footnotesize{Bases for $M/\{4\} \oplus \{4\}$ after $4$-contraction}};
                                
                   \draw (14,0.3) node{\footnotesize{$124~125~134~135~234~235$}};
                   \draw (10.55,2.4)--(10.85,2.4)--(10.85,2.7)--(10.55,2.7)--cycle;
                   \draw (10.55,2.1)--(10.85,2.1)--(10.85,2.4)--(10.55,2.4)--cycle;
                   \draw (13.8,2.2)--(14.1,2.2)--(14.1,2.5)--(13.8,2.5)--cycle;
                   \draw (14,2.3) node{\footnotesize{Bases for $M_1[\hspace{3mm}] \oplus M_2[\hspace{3mm}]$ after $(1,2)$-direct sum}};
                                               
         \end{tikzpicture}
      \end{center}
   \caption{$2$-deletion, $4$-contraction, and $(1,2)$-direct sum of matroids on $[P,Q]$ and their bases.}
   \label{i-R-C-D-LPM}
   \end{figure}      
    
   \begin{example}
   \label{facets-border-strip}
   The polytopes in Figure~\ref{i-R-C-D-facets} are facets of the polytope $\mathcal{P}(M[P,Q])$ corresponding to 
   $2$-deletion, $4$-contraction, and $3$-direct sum of the matroid $M[P,Q]$ in Figure~\ref{i-R-C-D-LPM}.
   \end{example}
   
    \begin{figure}[h]
       \begin{center}
          \begin{tikzpicture}[scale=.45]         
                                                 
                   \draw[ultra thick,orange] (0,1)--(1,0);
                   \draw (0,0)--(1,0)--(1,1)--(0,1)--cycle;
                   \draw (0,1)--(1,1)--(1,2)--(0,2)--cycle;
                   \draw (0,2)--(1,2)--(1,3)--(0,3)--cycle;
                   \draw (0.3,0.5) node{\scriptsize{$1$}};
                   \draw (1.3,0.5) node{\scriptsize{$3$}};
                   \draw (0.3,1.5) node{\scriptsize{$3$}};
                   \draw (1.3,1.5) node{\scriptsize{$4$}};
                   \draw (0.3,2.5) node{\scriptsize{$4$}};
                   \draw (1.3,2.5) node{\scriptsize{$5$}};
                   \draw (2.5,1.5) node{$\rightarrow$};
                   
                   \draw (5.5,3)--(4,0.7)--(5.5,0)--(7,0.7)--cycle;
                   \draw[dashed] (4,0.7)--(7,0.7);
                   \draw (5.5,3)--(5.5,0);  
                   \draw (5.5,3.3) node{\scriptsize{$134$}};
                   \draw (7.5,0.7) node{\scriptsize{$135$}};
                   \draw (3.4,0.7) node{\scriptsize{$145$}};
                   \draw (5.5,-0.5) node{\scriptsize{$345$}};                                    
                                      
                   \draw[ultra thick,orange] (11,2.5)--(12,1.5);
                   \draw (10,0.5)--(11,0.5)--(11,1.5)--(10,1.5)--cycle;
                   \draw (11,0.5)--(12,0.5)--(12,1.5)--(11,1.5)--cycle;
                   \draw (11,1.5)--(12,1.5)--(12,2.5)--(11,2.5)--cycle;
		   \draw (10,1.5)--(11,1.5)--(11,2.5)--(10,2.5)--cycle;
                   \draw (10.3,1) node{\scriptsize{$1$}};
                   \draw (11.3,1) node{\scriptsize{$2$}};
                   \draw (12.3,1) node{\scriptsize{$3$}};
                   \draw (10.3,2) node{\scriptsize{$2$}};
                   \draw (11.3,2) node{\scriptsize{$3$}};
                   \draw (12.3,2) node{\scriptsize{$5$}};
                   \draw (13.5,1.5) node{$\rightarrow$}; 
                   
                   \draw (16.5,3.3)--(15,1.5)--(16.5,0)--(18,1.8)--(16.5,3.3);
                   \draw (15,1.5)--(17,1.5)--(18,1.8);
                   \draw[dashed] (18,1.8)--(16,1.8)--(15,1.5);                    
                   \draw (16.5,3.3)--(17,1.5)--(16.5,0);
                   \draw[dashed] (16.5,3.3)--(16,1.8)--(16.5,0);
                   \draw (16.5,3.6) node{\scriptsize{$124$}};
                   \draw (15.5,2.2) node{\scriptsize{$245$}};
                   \draw (18.4,2.2) node{\scriptsize{$145$}};
                   \draw (14.6,1) node{\scriptsize{$234$}};
                   \draw (17.5,1) node{\scriptsize{$134$}};
                   \draw (16.5,-0.5) node{\scriptsize{$345$}};
                   \hspace{3mm}
                                      
                   \draw[orange, fill=orange] (21,2) circle (1ex);
                   \draw (20,0)--(21,0)--(21,1)--(20,1)--cycle;
                   \draw (20,1)--(21,1)--(21,2)--(20,2)--cycle;
                   \draw (21,2)--(22,2)--(22,3)--(21,3)--cycle;
                   \draw (20.3,0.5) node{\scriptsize{$1$}};
                   \draw (20.3,1.5) node{\scriptsize{$2$}};
                   \draw (21.3,0.5) node{\scriptsize{$2$}};
                   \draw (21.3,1.5) node{\scriptsize{$3$}};
                   \draw (21.3,2.5) node{\scriptsize{$4$}};
                   \draw (22.3,2.5) node{\scriptsize{$5$}};
                   \draw (23,1.5) node{$\rightarrow$};
                   
                   \draw (24.5,3)--(27,3)--(25.75,2.6)--(24.5,3);   
                   \draw (24.5,3)--(24.5,0.3)--(25.75,-0.1)--(27,0.3)--(27,3);
                   \draw (25.75,2.6)--(25.75,-0.1);
                   \draw[dashed] (24.5,0.3)--(27,0.3);
                   \draw (24,3.4) node{\scriptsize{$124$}};
                   \draw (24,-0.1) node{\scriptsize{$125$}};
                   \draw (27.5,3.4) node{\scriptsize{$134$}};
                   \draw (27.5,-0.1) node{\scriptsize{$135$}};
                   \draw (25.7,3) node{\scriptsize{$234$}};
                   \draw (25.7,-0.5) node{\scriptsize{$235$}};                     
                                                         
         \end{tikzpicture}
      \end{center}     
      \caption{$2$-deletion, $4$-contraction, and $3$-direct sum of matroids on $[P,Q]$ and corresponding facets.}
      \label{i-R-C-D-facets}
   \end{figure}
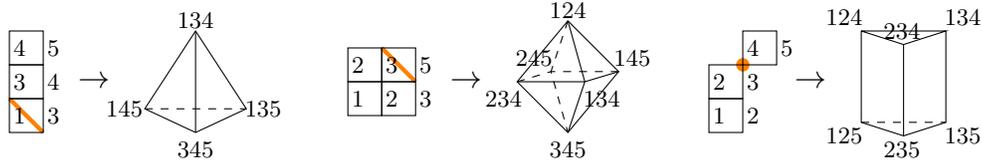
       
   In this section we focus on the properties of lattice path matroid polytopes corresponding to \emph{border strips}, 
   connected (non-empty) skew shapes with no $2 \times2 $ square. 
   Let $P = p_1 p_2 \cdots p_{m+r}$ and $Q = q_1 q_2 \cdots q_{m+r}$ be two lattice paths 
   from $(0,0)$ to $(m, r)$ with $P$ never going above $Q$. 
   Note that the region bounded by $P$ and $Q$ is a border strip if and only if 
   $p_i=q_i$ for $1 < i < m+r$, $p_1=q_{m+r}=\text{East step}$, and $q_1=p_{m+r}=\text{North step}$. 
   For such $P$ and $Q$ where $m+r > 2$, we define a lattice path $R = R(P,Q) = r_1 r_2 \cdots r_{m+r}$ as 
   $r_i = p_i (= q_i)$ for $1 < i < m+r$, $r_1 = r_2$, and $r_{m+r-1} = r_{m+r}$. 
   If $m+r = 2$, that means $P$ and $Q$ are lattice paths from $(0,0)$ to $(1,1)$, 
   a lattice path $R$ is defined as two consecutive North steps from $(0,0)$ to $(0,2)$. 
   Not only in the case $m+r = 2$, but in general $R$ is not needed to be a path from $(0,0)$ to $(m, r)$. 
   See Figure~\ref{P-Q-R}.
    
   \begin{figure}[h]
      \begin{center}
         \begin{tikzpicture}[scale=.45]
         
                   \draw[orange, pattern=north west lines, pattern color=orange] (0,0)--(1,0)--(0,1)--cycle;
                   \draw[orange, pattern=north west lines, pattern color=orange] (5,7)--(4,7)--(5,6)--cycle;
                   \filldraw[orange, fill=orange] (0,3)--(1,3)--(2,2)--(1,2)--cycle;
                   \filldraw[orange, fill=orange] (4,6)--(4,5)--(5,4)--(5,5)--cycle;
                   \draw (0,0)--(1,0)--(1,1)--(0,1)--cycle;
                   \draw (0,1)--(1,1)--(1,2)--(0,2)--cycle;
                   \draw (0,2)--(1,2)--(1,3)--(0,3)--cycle;
		   \draw (1,2)--(2,2)--(2,3)--(1,3)--cycle;
                   \draw (2,2)--(3,2)--(3,3)--(2,3)--cycle;
                   \draw (3,2)--(4,2)--(4,3)--(3,3)--cycle;
                   \draw (3,3)--(4,3)--(4,4)--(3,4)--cycle;
                   \draw (4,3)--(5,3)--(5,4)--(4,4)--cycle;
		   \draw (4,4)--(5,4)--(5,5)--(4,5)--cycle;
                   \draw (4,5)--(5,5)--(5,6)--(4,6)--cycle;
                   \draw (4,6)--(5,6)--(5,7)--(4,7)--cycle;
                   \draw (1.45,-0.7) node{\footnotesize{$E = p_1$}};
                   \draw (-1.5,0.33) node{\footnotesize{$q_1 = N$}};
                   \draw (-1.3,3.4) node{\footnotesize{$p_4 = q_4 = E$}};
                   \draw (7.6,4.5) node{\footnotesize{$N = p_{10} = q_{10}$}};
                   \draw (3.5,7.4) node{\footnotesize{$q_{12} = E$}};
                   \draw (6.6,6.5) node{\footnotesize{$N = p_{12}$}};                                   
                   \draw (2.5,-2.3) node{\footnotesize{$[P,Q]$}};
                                                         
                   \draw (13,0)--(13,3)--(16,3)--(16,4)--(17,4)--(17,8);
                   \draw[ultra thick,orange] (13,3)--(14,3);
                   \draw[ultra thick,orange] (17,5)--(17,6);
                   \fill (13,0) circle(4pt);
                   \fill (13,1) circle(4pt);
                   \fill (13,2) circle(4pt);
                   \fill (13,3) circle(4pt);
                   \fill (14,3) circle(4pt);
                   \fill (15,3) circle(4pt);
                   \fill (16,3) circle(4pt);
                   \fill (16,4) circle(4pt);
                   \fill (17,4) circle(4pt);
                   \fill (17,5) circle(4pt);
                   \fill (17,6) circle(4pt);
                   \fill (17,7) circle(4pt);
                   \fill (17,8) circle(4pt);
                   \draw (12.5,3.5) node{\footnotesize{$r_4 = E$}};
                   \draw (19.9,7.4) node{\footnotesize{$N = r_{12} (= r_{11})$}};
                   \draw (15.6,0.4) node{\footnotesize{$N = r_1 (= r_2)$}};
                   \draw (18.7,5.4) node{\footnotesize{$N = r_{10}$}};
                   \draw (15.5,-2.3) node{\footnotesize{$R(P,Q)$}};        
                                               
         \end{tikzpicture}
      \end{center}
   \caption{Border strip $[P,Q]$ from $(0,0)$ to $(5,7)$ and lattice path $R(P,Q)$ from $(0,0)$ to $(4,8)$.}
   \label{P-Q-R}
   \end{figure}
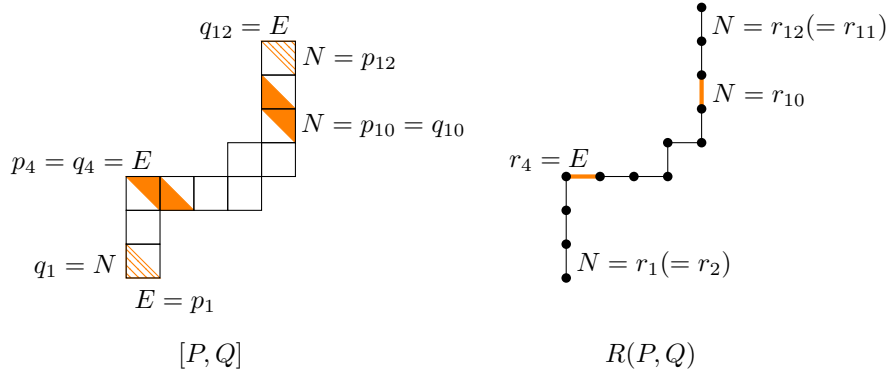     

   For a lattice path $R = R(P,Q)$ we define three sets $\mathcal{D}(R)$, $\mathcal{C}(R)$, $\mathcal{S}(R)$ as follows:
   \begin{align*}
   \mathcal{D}(R) &= \{i \text{-deletion of $M[P,Q]$} : r_i = \text{East step}\},\\
   \mathcal{C}(R) &= \{i \text{-contraction of $M[P,Q]$} : r_i = \text{North step}\},
   \text{and}\\
   \mathcal{S}(R) &= \{i \text{-direct sum of $M[P,Q]$} : r_i \neq r_{i+1}\}. 
   \end{align*}
   
   \noindent
   For each element of 
   $\mathcal{D}(R(P,Q)) \cupdot \mathcal{C}(R(P,Q)) \cupdot \mathcal{S}(R(P,Q))$ 
   we have the corresponding border strip. See Figure~\ref{R-C-D-LPMP}.

   \begin{figure}[h]
      \begin{center}
         \begin{tikzpicture}[scale=.30]
         
                   \filldraw[orange, fill=orange] (3,2)--(4,1)--(4,2)--(3,3)--cycle;
                   \draw (0,0)--(1,0)--(1,1)--(0,1)--cycle;
                   \draw (0,1)--(1,1)--(1,2)--(0,2)--cycle;
                   \draw (1,1)--(2,1)--(2,2)--(1,2)--cycle;
		   \draw (2,1)--(3,1)--(3,2)--(2,2)--cycle;
                   \draw (3,1)--(4,1)--(4,2)--(3,2)--cycle;
                   \draw (3,2)--(4,2)--(4,3)--(3,3)--cycle;
                   \draw (3,3)--(4,3)--(4,4)--(3,4)--cycle;
                   \draw (4,3)--(5,3)--(5,4)--(4,4)--cycle;
                   
                   \draw (6.5,2) node{$\rightarrow$};
                   \draw (6.5,-2) node{\footnotesize{$6$-contraction ($r_6 = N$)}};   
                   
                   \draw[ultra thick,orange] (11,2)--(12,1);
                   \draw (8,0)--(9,0)--(9,1)--(8,1)--cycle;
                   \draw (8,1)--(9,1)--(9,2)--(8,2)--cycle;
                   \draw (9,1)--(10,1)--(10,2)--(9,2)--cycle;
                   \draw (10,1)--(11,1)--(11,2)--(10,2)--cycle;
                   \draw (11,1)--(12,1)--(12,2)--(11,2)--cycle;
                   \draw (11,2)--(12,2)--(12,3)--(11,3)--cycle;
                   \draw (12, 2)--(13,2)--(13,3)--(12,3)--cycle;
                   
                   \filldraw[orange, fill=orange] (17,2)--(18,2)--(19,1)--(18,1)--cycle;
                   \draw (16,0)--(17,0)--(17,1)--(16,1)--cycle;
                   \draw (16,1)--(17,1)--(17,2)--(16,2)--cycle;
                   \draw (17,1)--(18,1)--(18,2)--(17,2)--cycle;
		   \draw (18,1)--(19,1)--(19,2)--(18,2)--cycle;
                   \draw (19,1)--(20,1)--(20,2)--(19,2)--cycle;
                   \draw (19,2)--(20,2)--(20,3)--(19,3)--cycle;
                   \draw (19,3)--(20,3)--(20,4)--(19,4)--cycle;
                   \draw (20,3)--(21,3)--(21,4)--(20,4)--cycle;
                   
                   \draw (22.5,2) node{$\rightarrow$};
                   \draw (22.5,-2) node{\footnotesize{$4$-deletion ($r_4 = E$)}};  
                   
                   \draw[ultra thick,orange] (25,2)--(26,1);
                   \draw (24,0)--(25,0)--(25,1)--(24,1)--cycle;
                   \draw (24,1)--(25,1)--(25,2)--(24,2)--cycle;
                   \draw (25,1)--(26,1)--(26,2)--(25,2)--cycle;
		   \draw (26,1)--(27,1)--(27,2)--(26,2)--cycle;
                   \draw (26,2)--(27,2)--(27,3)--(26,3)--cycle;
                   \draw (26,3)--(27,3)--(27,4)--(26,4)--cycle;
                   \draw (27,3)--(28,3)--(28,4)--(27,4)--cycle;
                                                                          
                   \filldraw[orange, fill=orange] (31,1)--(31,2)--(32,2)--(32,1)--cycle;
                   \draw (31,0)--(32,0)--(32,1)--(31,1)--cycle;
                   \draw (31,1)--(32,1)--(32,2)--(31,2)--cycle;
                   \draw (32,1)--(33,1)--(33,2)--(32,2)--cycle;
		   \draw (33,1)--(34,1)--(34,2)--(33,2)--cycle;
                   \draw (34,1)--(35,1)--(35,2)--(34,2)--cycle;
                   \draw (34,2)--(35,2)--(35,3)--(34,3)--cycle;
                   \draw (34,3)--(35,3)--(35,4)--(34,4)--cycle;
                   \draw (35,3)--(36,3)--(36,4)--(35,4)--cycle;
                   
                   \draw (37.5,2) node{$\rightarrow$};
                   \draw (37.5,-2) node{\footnotesize{$2$-direct sum ($r_2 \neq r_3$)}};   
                   
                   \draw[orange, fill=orange] (40,1) circle (1ex);
                   \draw (39,0)--(40,0)--(40,1)--(39,1)--cycle;
                   \draw (40,1)--(41,1)--(41,2)--(40,2)--cycle;
		   \draw (41,1)--(42,1)--(42,2)--(41,2)--cycle;
                   \draw (42,1)--(43,1)--(43,2)--(42,2)--cycle;
                   \draw (42,2)--(43,2)--(43,3)--(42,3)--cycle;
                   \draw (42,3)--(43,3)--(43,4)--(42,4)--cycle;
                   \draw (43,3)--(44,3)--(44,4)--(43,4)--cycle;
                        
                   \hspace{10mm}                                                   
                                               
                   \draw (0,8)--(1,8)--(1,9)--(0,9)--cycle;
                   \draw (0,9)--(1,9)--(1,10)--(0,10)--cycle;
                   \draw (1,9)--(2,9)--(2,10)--(1,10)--cycle;
                   \draw (2,9)--(3,9)--(3,10)--(2,10)--cycle;
                   \draw (3,9)--(4,9)--(4,10)--(3,10)--cycle;
                   \draw (3,10)--(4,10)--(4,11)--(3,11)--cycle;
                   \draw (3,11)--(4,11)--(4,12)--(3,12)--cycle;
		   \draw (4,11)--(5,11)--(5,12)--(4,12)--cycle;
                   \draw (2.5,6.3) node{\footnotesize{$[P,Q]$}};
                                                         
                   \draw (12,8)--(12,10)--(15,10)--(15,12)--(17,12);
                   \fill (12,8) circle(4pt);
                   \fill (12,9) circle(4pt);
                   \fill (12,10) circle(4pt);
                   \fill (13,10) circle(4pt);
                   \fill (14,10) circle(4pt);
                   \fill (15,10) circle(4pt);
                   \fill (15,11) circle(4pt);
                   \fill (15,12) circle(4pt);
                   \fill (16,12) circle(4pt);
                   \fill (17,12) circle(4pt);
                   \draw (15,6.3) node{\footnotesize{$R(P,Q)$}};                   
                   
                   \draw (31,10) node{\footnotesize{$R(P,Q) = r_1 r_2 r_3 r_4 r_5 r_6 r_7 r_8 r_9$}};
                   \draw (33,8.5) node{\footnotesize{$= N N E E E N N E E$}};
                                      
         \end{tikzpicture}
      \end{center}
   \caption{Border strips corresponding to deletion, contraction, and direct sum of $M[P,Q]$.}
   \label{R-C-D-LPMP}
   \end{figure}
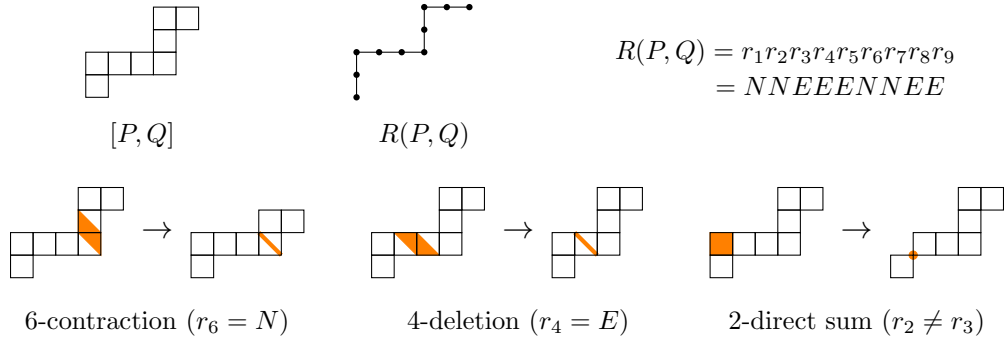 
      
   Note that the dimension of $\mathcal{P}(M[P,Q]) = \dim(\mathcal{P}(M[P,Q])) = m+r-1$ 
   for lattice paths $P$ and $Q$ from $(0,0)$ to $(m, r)$ with $P$ never going above $Q$ 
   where the region bounded by $P$ and $Q$ is a border strip 
   since $P$ and $Q$ satisfy $k=2$ in Corollary~\ref{dimension-of-matroid-polytope}.

   \begin{lemma}
   \label{facets-of-matroid-polytope}
   The set of facets of lattice path matroid polytope $\mathcal{P}(M[P,Q])$, 
   where the region bounded by $P$ and $Q$ is a border strip, 
   has a one-to-one correspondence with 
   $\mathcal{D}(R(P,Q)) \cupdot \mathcal{C}(R(P,Q)) \cupdot \mathcal{S}(R(P,Q))$. 
   \end{lemma}
   
   \begin{proof}
   For the polytope $\mathcal{P}(M[P,Q]) = \mathcal{P}(\mathcal{B}(M)) 
   = \conv \{ e_B  = e_{b_1} + \cdots + e_{b_r} : B \in \mathcal{B} \}$, 
   take a subset $S^{-i} = \conv\{e_{B} : i \notin B \in \mathcal{B}\}$ 
   on the hyperplane $x_i = 0$ in $\reals^{m+r-k+2}$, 
   which is corresponding to $i$-deletion in $\mathcal{D}(R)$. 
   All the vertices $e_B \in \mathcal{P}(M[P,Q]) - S^{-i}$ lie on the half-space $x_i > 0$ 
   since their $i\text{-}th$ coordinates are $1 (> 0)$. 
   The dimension of $S^{-i}$ is $(m-1)+r-k+1 = m+r-k$ 
   since $k$ and $r$ are fixed and the only width $m$ dropped by $1$ during the $i$-deletion. 
   Hence, $S^{-i}$ is a facet of $\mathcal{P}(M[P,Q]).$
   
   Similarly for $i$-contraction in $\mathcal{C}(R)$
   if we take a subset $S^{+i} = \conv\{e_{B} : i \in B \in \mathcal{B}\}$ 
   on the hyperplane $x_i = 1$ in $\reals^{m+r-k+2}$, 
   all the vertices of $\mathcal{P}(M[P,Q]) - S^{+i}$ have $i\text{-}th$ coordinate $0 (< 1)$ and 
   lie on the half-space of the hyperplane $x_i < 1$. 
   After the $i-$contraction  the height $r$ dropped by $1$, while $m$ and $k$ are fixed, 
   and the dimension of $S^{+i}$ is $m+(r-1)-k+1 = m+r-k$. 
   Hence, $S^{+i}$ is also a facet of $\mathcal{P}(M[P,Q])$.

   For $i$-direct sum in $\mathcal{D}(R)$, without loss of generality, 
   we may assume that $r_i$ is East step and $r_{i+1}$ is North step. 
   That means a direct sum occurs at the point $(p,q)$ of the path $Q$ where $p+q = i$. 
   Take a subset $S^{i} = \conv\{e_{B} : |[i] \cap B| = q, B \in \mathcal{B}\}$ 
   of  $\mathcal{P}(M[P,Q])$. 
   Then, $S^i$ lies on the hyperplane $x_1+x_2+ \cdots + x_i = q$ and 
   the other points on the half-space $x_1+x_2+ \cdots + x_i < q$. 
   The dimension of $S^{i}$ is $m+r-(k+1)+1 = m+r-k$ 
   since we have the same end points and get one more intersection point after $i$-direct sum. 
   Hence, $S^{i}$ is a facet of $\mathcal{P}(M[P,Q])$. 
   
   For the other direction, 
   to show is all the facets of $\mathcal{P}(M[P,Q])$ are on the hyperplanes corresponding to
   $\mathcal{D}(R(P,Q)) \cupdot \mathcal{C}(R(P,Q)) \cupdot \mathcal{S}(R(P,Q))$. 
   Suppose a polytope $\mathcal{P}(M[P,Q])$ has a facet not lying on the hyperplane 
   from $\mathcal{D}(R(P,Q)) \cupdot \mathcal{C}(R(P,Q)) \cupdot \mathcal{S}(R(P,Q))$. 
   That means hyperplanes corresponding to 
   $\mathcal{D}(R(P,Q)) \cupdot \mathcal{C}(R(P,Q)) \cupdot \mathcal{S}(R(P,Q))$ 
   do not generate $\mathcal{P}(M[P,Q])$ 
   in its affine hull  $x_1 + x_2 + \cdots + x_{m+r-k+2} = r$. 
   Then, there exists a point 
   $x = (x_1, x_2, \ldots, x_{m+r-k+2}) \in \reals^{m+r-k+2} - \mathcal{P}(M[P,Q])$, 
   located in the intersection of all the half-spaces mentioned in the previous paragraphs 
   and $x_1 + x_2 + \cdots + x_{m+r-k+2} = r$. 
   
   Since $x$ satisfies above conditions,
   for any maximal sequence of consecutive $E$'s of $R$, $r_u r_{u+1} \cdots r_{u+v}$,
   we have $x_i \geq 0$ for $i \in [u, u+v]$ and 
   $x_1 + x_2 + \cdots + x_u \geq N(R,u)$ and 
   $x_1 + x_2 + \cdots + x_{u+v} \leq N(R, u) + 1$ 
   where $N(R, i)$ is the number of North steps until $i$th step of $R$. 
   Then, $x_i \leq 1$ for $i \in [u, u+v]$, and
   this implies $0 \leq x_i \leq 1$ if $r_i = E$. 
   Similarly, for any maximal sequence $r_u r_{u+1} \cdots r_{u+v}$ of consecutive $N$'s of $R$, 
   we obtain $0 \leq x_i \leq 1$ for $i \in [u, u+v]$, and 
   this  means $0 \leq x_i \leq 1$ if $r_i = N$. 
   Hence, we get $0 \leq x_i \leq 1$ for all $i$ in $[m+r]$. 
   
   For each sequence $r_j r_{j+1}$ of $R$ such that $r_j r_{j+1} = NE$ or $j = m+r$, 
   two conditions, $x_j \leq 1$ and $x_1 + x_2 + \cdots + x_j \geq N(R, j) = N(P, j)$, are given. 
   Hence, it follows that 
   $x_1 + x_2 + \cdots + x_{j-1} \geq N(R, j) - 1 = N(R, j-1) = N(P, j-1)$. 
   If $r_{j-1}$ is a North step, 
   $x_{j-1} \leq 1$ and $x_1 + x_2 + \cdots + x_{j-2} \geq N(R, j-1) - 1 = N(R, j-2) = N(P, j-2)$. 
   If $r_{j-1}$ is an East step, 
   $x_1 + x_2 + \cdots + x_{j-2} \geq N(R, h) = N(R, j-2) = N(P, j-2)$  
   where $r_h r_{h+1}$ is a previous $NE$ sequence of $R$ or $h=1$. 
   After checking each East step $r_{j-k} $ for $1 \leq k \leq j-h-1$, 
   we  have $x_1 + x_2 + \cdots + x_i \geq N(P, i)$ for all $i$ in $[m+r]$. 
   If we apply a similar way to each subsequence $r_j r_{j+1}$ of $R$ 
   such that $r_j r_{j+1} = EN$ or $j=1$ 
   where $x_{j} \geq 0$ and $x_1 + x_2 + \cdots + x_j \leq N(R, j) + 1 = N(Q, j)$, 
   we also get $x_1 + x_2 + \cdots + x_i \leq N(Q,i)$ for all $i$ in $[m+r]$. 
   
   Hence, we conclude 
   $0 \leq x_i \leq 1$ and 
   $N(P,i) \leq x_1 + x_2 + \cdots + x_i \leq N(Q,i)$ for all $i \in [m+r]$. 
   This is a contradiction to the fact that $x$ is not a point in $\mathcal{P}(M[P,Q])$.
   Therefore, hyperplanes corresponding to 
   $\mathcal{D}(R(P,Q)) \cupdot \mathcal{C}(R(P,Q)) \cupdot \mathcal{S}(R(P,Q))$ generate 
   $\mathcal{P}(M[P,Q])$ in its affine hull  $x_1 + x_2 + \cdots + x_{m+r-k+2} = r$, and 
   all the facets of $\mathcal{P}(M[P,Q])$ are on the hyperplanes corresponding to 
   $\mathcal{D}(R(P,Q)) \cupdot \mathcal{C}(R(P,Q)) \cupdot \mathcal{S}(R(P,Q))$.
   \end{proof}
       
   \begin{corollary}
   \label{number-of-facets}
   For lattice paths $P$ and $Q$ from $(0,0)$ to $(m, r)$ 
   with the region bounded by $P$ and $Q$ as a border strip, 
   the number of facets of $\mathcal{P}(M[P,Q])$ is 
   $m+r+d$ where $d$ is the number of outside corners of the region $[P,Q]$.
   \end{corollary}   

   Not only facets of lattice path matroid polytope, 
   we will find a one to one corresponding set 
   for all the faces of lattice path matroid polytope.

   If we consider a lattice path $R(P,Q)$ as a sequence on $\{E, N\}^{m+r}$ and 
   cut the sequence $R$ at every direct sum position, 
   we may get $d+1$ subsequences where $d = |\mathcal{S}(R(P,Q))| = \text{the number of corners of } R(P,Q)$. 
   Let $(S_1, S_2, \ldots, S_{d+1})$ be a set partition of $\mathcal{D}(R(P,Q)) \cupdot \mathcal{C}(R(P,Q))$ with $S_i$ 
   corresponding to $i$th subsequence of $R$. 
   Define a set $S_i^{L} = S_i \cup \{(i-1) \text{th direct sum}\}$ for $1 < i \leq d+1$ and $S_1^L = S_1$. 
   Similarly, define a set $S_i^R = S_i \cup \{i \text{th direct sum}\}$ for $1 \leq i < d+1$ and $S_{d+1}^R = S_{d+1}$.

   For a subset $T$ of $\mathcal{D}(R(P,Q)) \cupdot \mathcal{C}(R(P,Q)) \cupdot \mathcal{S}(R(P,Q))$ 
   we introduce the following three conditions:

   \begin{enumerate}
      \item[(C1)]
      $S_i^R \nsubseteq T$ for $2 \leq i \leq d+1$.
      \item[(C2)]
      For each sequence $S_i^L, S_{i+1}, \ldots, S_{d}$ where $1 \leq i \leq d$,
      a subsequence $S_i^L, S_{i+1}, \ldots, S_{j}$ can be included in $T$ 
      if and only if the subsequence has an even-length and $S_{j+1} \nsubseteq T$.
      \item[(C3)]
      A sequence $S_i^L, S_{i+1}, \ldots, S_{d+1}-\{(m+r)\text{-deletion}, (m+r)\text{-contraction}\}$
      can be included in $T$ for $1 \leq i \leq d$ 
      if and only if the sequence has an even-length.
   \end{enumerate}

   \begin{example}
   \label{3-rules}
   For a lattice path $R(P,Q) = E^2N^2ENE^3NEN^4$,
   we have $d = 7$ and
   $S_1 = \{1 \text{-deletion}, 2 \text{-deletion}\}$,
   $S_2 = \{3 \text{-contraction}, 4 \text{-contraction}\}$,
   $S_3 = \{5 \text{-deletion}\},$ $\ldots ,$
   $S_8 = \{12 \text{-contraction}, \ldots, 15 \text{-contraction}\}$.
   Figure~\ref{3-conditions} shows condition (C1), (C2), and (C3).
   
   (a) The orange arrows represent $7$ sets,
   $S_2^R, S_3^R, \ldots, S_7^R, \text{and}~ S_8^R (=S_8)$,
   which are not included in $T$ by (C1). 
   As an example, $T$ cannot be a $4$-subset such as 
   $\{3 \text{-contraction}, 4 \text{-contraction}, 4 \text{-direct sum}, 5 \text{-deletion}\}$
   since it contains $S_2^R$.
   
   (b) The orange arrow is a sequence $S_2^L, S_{3}, \ldots, S_{7}$.
   By condition (C2), $\{2 \text{-direct sum}, 3 \text{-contraction}, 4 \text{-contraction}\}$
   cannot be $T$ since it contains $S_2^L$, but not $S_3$. 
   However, $\{2 \text{-direct sum}, 3 \text{-contraction}, 4 \text{-contraction}, 5 \text{-deletion}\}$
   can be $T$ since it includes $S_2^L$ and $S_3$, but not $S_4$.  
   
   (c) As we see the green and blue arrows in Figure~\ref{3-conditions}(b), 
   we need the third condition (C3) if $j = d+1$. 
   The green arrow represents the last condition that 
   $T$ can include a sequence $S_3^L, S_4, S_5, S_6, S_7, S_8-\{15 \text{-contraction}\}$, 
   and the last condition for the blue arrow is that
   a sequence $S_1^L, S_2, S_3, S_4, S_5, S_6, S_7, S_8-\{15 \text{-contraction}\}$ can be contained in $T$.
   
   \begin{figure}[h]
      \begin{center}
         \begin{tikzpicture}[scale=.35]
         
                   \draw (0,0)--(2,0)--(2,2)--(3,2)--(3,3)--(6,3)--(6,4)--(7,4)--(7,8);
                   \fill (0,0) circle(6pt);
                   \fill (1,0) circle(6pt);
                   \fill (2,0) circle(6pt);
                   \fill (2,1) circle(6pt);
                   \fill (2,2) circle(6pt);
                   \fill (3,2) circle(6pt);
                   \fill (3,3) circle(6pt);
                   \fill (4,3) circle(6pt);
                   \fill (5,3) circle(6pt);
                   \fill (6,3) circle(6pt);
                   \fill (6,4) circle(6pt);
                   \fill (7,4) circle(6pt);
                   \fill (7,5) circle(6pt);
                   \fill (7,6) circle(6pt);
                   \fill (7,7) circle(6pt);
                   \fill (7,8) circle(6pt);
                   \draw (4,-1.6) node{\footnotesize{(a) $S_i^R$ for $2 \leq i \leq 8$}};
                   
                   \draw[orange,thick,->] (7,8) -- (7.3,4);
                   \fill (7,8) circle(4pt);
                   \draw[orange,thick,->] (7,4) -- (6,4.3);
                   \fill[orange] (7,4) circle(4pt);
                   \draw[orange,thick,->] (6,4) -- (6.3,3);
                   \fill[orange] (6,4) circle(4pt);
                   \draw[orange,thick,->] (6,3) -- (3,3.3);
                   \fill[orange] (6,3) circle(4pt);
                   \draw[orange,thick,->] (3,3) -- (3.3,2);
                   \fill[orange] (3,3) circle(4pt);
                   \draw[orange,thick,->] (3,2) -- (2,2.3);
                   \fill[orange] (3,2) circle(4pt);
                   \draw[orange,thick,->] (2,2) -- (2.3,0);
                   \fill[orange] (2,2) circle(4pt);
                   
                   \hspace{-5mm}
                   
                   \draw (15,0)--(17,0)--(17,2)--(18,2)--(18,3)--(21,3)--(21,4)--(22,4)--(22,8);
                   \fill (15,0) circle(6pt);
                   \fill (16,0) circle(6pt);
                   \fill (17,0) circle(6pt);
                   \fill (17,1) circle(6pt);
                   \fill (17,2) circle(6pt);
                   \fill (18,2) circle(6pt);
                   \fill (18,3) circle(6pt);
                   \fill (19,3) circle(6pt);
                   \fill (20,3) circle(6pt);
                   \fill (21,3) circle(6pt);
                   \fill (21,4) circle(6pt);
                   \fill (22,4) circle(6pt);
                   \fill (22,5) circle(6pt);
                   \fill (22,6) circle(6pt);
                   \fill (22,7) circle(6pt);
                   \fill (22,8) circle(6pt);
                   \draw (20,-1.6) node{\footnotesize{(b) $S_i^L, S_{i+1}, \ldots, S_{j}$ for $1 \leq i \leq 7$}};
                   
                   \draw[blue,thick,->] (15,0)--(17.3,-0.3)--(17.3,1.7)--(18.3,1.7)--(18.3,2.7)--(21.3,2.7)--(21.3,3.7)--(22.3,3.7)--(22.3,7);
                   \fill (15,0) circle(4pt);
                   \draw[orange,thick,->] (17,0)--(16.4,2.6)--(17.4,2.6)--(17.4,3.6)--(20.4,3.6)--(20.4,4.6)--(22,4.6);
                   \fill[orange] (17,0) circle(4pt);
                   \draw[olive,thick,->] (17,2)--(17.7,2.3)--(17.7,3.3)--(20.7,3.3)--(20.7,4.3)--(21.7,4.3)--(21.7,7);
                   \fill[olive] (17,2) circle(4pt);                  
                   \draw[gray,thick,dashed] (22,6) circle (2);
                   \draw[gray,->,snake=snake,segment amplitude=.4mm,segment length=2mm,line after snake=1mm] (24,6)--(31,6);               
                   
                   \draw (30,0.5)--(34.5,0.5)--(34.5,2)--(36,2)--(36,8);
                   \fill (30,0.5) circle(6pt);
                   \fill (31.5,0.5) circle(6pt);
                   \fill (33,0.5) circle(6pt);
                   \fill (34.5,0.5) circle(6pt);
                   \fill (34.5,2) circle(6pt);
                   \fill (36,2) circle(6pt);
                   \fill (36,3.5) circle(6pt);
                   \fill (36,5) circle(6pt);
                   \fill (36,6.5) circle(6pt);
                   \fill (36,8) circle(6pt); 
                   \draw (34,-1.6) node{\footnotesize{(c) $S_{j} = S_8$}};
                   
                   \draw[gray,thick,dashed] (34,6.5)--(38,6.5);
                   \draw[gray,thick,dashed] (35,5) circle (4);
                   \draw[orange,thick,->] (30,1.1)--(33.9,1.1)--(33.9,2.6)--(36,2.6);
                   \draw[blue,thick,->] (30,0.2)--(34.8,0.2)--(34.8,1.7)--(36.3,1.7)--(36.3,6.5);
                   \draw[olive,thick,->] (30,0.8)--(34.2,0.8)--(34.2,2.3)--(35.7,2.3)--(35.7,6.5);
                                       
         \end{tikzpicture}
      \end{center}
   \caption{Corners on $R(P,Q)$.}
   \label{3-conditions}
   \end{figure}
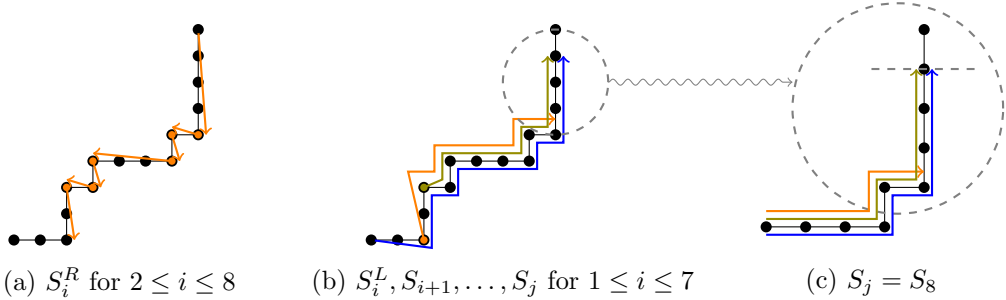
   \end{example}

   \begin{theorem}
   \label{faces(m+r-1-l)-of-matroid-polytope}
      For lattice paths $P$ and $Q$ from $(0,0)$ to $(m, r)$ with $P$ never going above $Q$ and 
      the region bounded by $P$ and $Q$ being a border strip, 
      the set of $(m+r-1-t)$-dimensional faces of the lattice path matroid polytope $\mathcal{P}(M[P,Q])$ has a one-to-one correspondence with 
      $t$-subsets of $\mathcal{D}(R(P,Q)) \cupdot \mathcal{C}(R(P,Q)) \cupdot \mathcal{S}(R(P,Q))$ satisfying condition (C1), (C2), and (C3).
   \end{theorem}

   \begin{proof}
   Note that 
   $(m+r-1-t)$-dimensional faces of lattice path matroid polytope $\mathcal{P}(M[P,Q])$ are 
   facets of $(m+r-t)$-dimensional faces of $\mathcal{P}(M[P,Q])$. 
   First, to show is each construction, 
   in $t$-subsets of 
   $\mathcal{D}(R(P,Q)) \cupdot \mathcal{C}(R(P,Q)) \cupdot \mathcal{S}(R(P,Q))$ 
   satisfying condition (C1), (C2), and (C3), 
   is possible after other constructions have been done 
   while the dimensions are being dropped from $m+r-1$ to $m+r-1-t$. 
   That means, 
   any $s$ structures from $t$-subsets drop 
   the dimension of $\mathcal{P}(M[P,Q])$  by $s$ exactly, where $1 \leq s \leq t$. 
   It is not hard to check using the similar steps in Lemma~\ref{facets-of-matroid-polytope}.
   
   For the other direction, 
   suppose that $S_i^R \subseteq T$ for some $i~(2 \leq i \leq d+1)$ in (C1).
   Since $i$-direct sum drops more than $1$ dimension of the polytope 
   obtained by $|S_i|$ contractions (or deletions) in $S_i^R$, 
   the dimension of faces after applying all the constructions in $S_i^R$ is 
   less than $m+r-1-|S_i^R|$. 
   Similarly, suppose 
   an odd length sequence $S_i^L, S_{i+1}, \ldots, S_{j} \subseteq T$ and $S_{j+1} \nsubseteq T$ 
   for some $i$ and $j$ such that $1 \leq i < j \leq d$ in (C2). 
   Then, the dimension of faces 
   after all the constructions in $S_i^L, S_{i+1}, \ldots, S_{j} \subseteq T$ are applied is 
   less than $m+r-1-(|S_i^L| + |S_{i+1}| + \cdots + |S_{j}|)$
   since $(i-1)$-direct sum drops more than $1$ dimension of the polytope obtained 
   after $|S_i| + |S_{i+1}| + \cdots + |S_{j}|$ contractions and deletions. 
   As the same way,  $(i-1)$-direct sum also drops more than $1$ dimension of the polytope in (C3). 
   Therefore, the set of $(m+r-1-t)$-dimensional faces of lattice path matroid polytope $\mathcal{P}(M[P,Q])$ has a one-to-one correspondence with 
   $t$-subsets of $\mathcal{D}(R(P,Q)) \cupdot \mathcal{C}(R(P,Q)) \cupdot \mathcal{S}(R(P,Q))$ satisfying condition (C1), (C2), and (C3).  
   \end{proof}

\section{General case}
\label{sec-skew-shape}
   
   In this section, we consider the region $[P,Q]$ as a connected (non-empty) skew shape, and 
   the faces of a lattice path matroid polytope $\mathcal{P}(M[P,Q])$ are described 
   in terms of certain tiled subregions without $2 \times 2$ rectangles inside $[P,Q]$. 
   Note that, even if the subregion is not allowed, 
   $[P,Q]$ may contain $2 \times 2$ rectangles unlike previous sections.
      
   We begin with the following proposition 
   which is a generalization of Lemma~\ref{facets-of-matroid-polytope}. 
   Since $[P,Q]$ is a connected skew shape now, 
   $P$ and $Q$ have only two intersection points $(0,0)$ and $(m,r)$,
   and $R$ from $[P,Q]$ is a sequence corresponding to a border strip from $(0,0)$ to $(m,r)$ contained in $[P,Q]$. 
   Its proof is omitted since it is similar to the proof of Lemma~\ref{facets-of-matroid-polytope}.
   
   \begin{proposition}
   \label{facets-of-matroid-polytope-general-case}
      The set of facets of lattice path matroid polytope $\mathcal{P}(M[P,Q])$ has one-to-one correspondence with 
      the disjoint union of the following three sets:
      \begin{align*}
         \mathcal{D}(P,Q) &= \{i \text{-deletion} : r_i = \text{East step for some $R$ from } [P,Q]\},\\
         \mathcal{C}(P,Q) &= \{i \text{-contraction} : r_i = \text{North step for some $R$ from } [P,Q]\}, \text{and}\\
         \mathcal{S}(P,Q) &= \{(p,q) \text{-direct sum} : (p, q) \text{ is an outside corner of } [P,Q]\}.
      \end{align*}
   \end{proposition}
   
   \noindent
   Note that if the region $[P,Q]$ is a border strip, 
   the above three sets coincide with 
   $\mathcal{D}(R), \mathcal{C}(R)$, and $\mathcal{S}(R)$ in the previous section, respectively.
      
   Since a face of $\mathcal{P}(M[P,Q])$ is 
   a facet of a one-higher-dimensional face of $\mathcal{P}(M[P,Q])$, 
   Proposition~\ref{facets-of-matroid-polytope-general-case} implies that 
   all the faces of $\mathcal{P}(M[P,Q])$ are obtained 
   after applying deletions, contractions, and direct sums.
   Note that a set of matroid constructions used to generate a face of $\mathcal{P}(M[P,Q])$ 
   may not be uniquely determined.
   
   Before we give a description of faces of $\mathcal{P}(M[P,Q])$, 
   we need to define several notions.
   We label each unit box having points $(i,j)$ and $(i+1,j+1)$ inside the region $[P,Q]$ 
   with $i+j+1$, 
   and let $(i,j)$ be the \emph{starting point} and 
   $(i+1,j+1)$ be the \emph{ending point} of the box. 
   A \emph{block} is a border strip located inside $[P,Q]$, 
   and we may consider a block as a tableau with labels of boxes in the block. 
   The starting point and ending point of a block are 
   the starting point of the smallest labelled box and 
   the ending point of the largest labelled box contained in the block respectively.
   The \emph{clones} inside the region $[P,Q]$ are blocks such that 
   they are the same as tableaux and distinguishable only by the difference in their positions.
   A block is a clone of itself.
   In Figure~\ref{clone}, 
   the block with the starting point $(1,0)$ and the ending point $(3,2)$
   is a clone of
   the block with the starting point $(0,1)$ and the ending point $(2,3)$, and vice versa.
   
   \begin{figure}[h]
      \begin{center}
         \includegraphics[width = 0.15\textwidth]{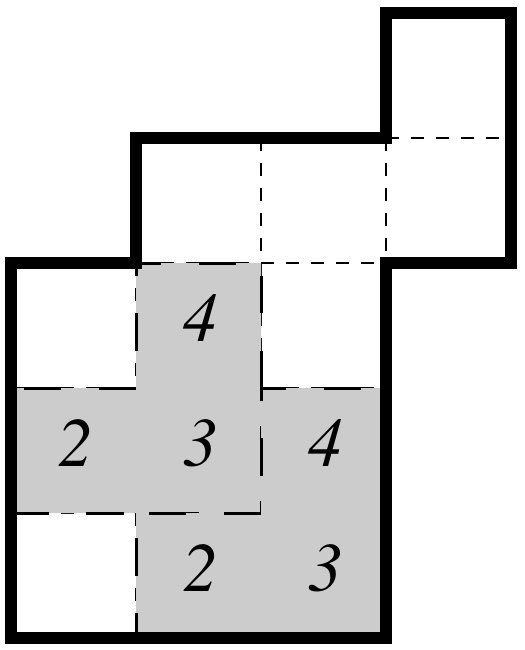}
      \end{center}
   \caption{Clones labeled by $2$-$3$-$4$ in $[P,Q]$}
   \label{clone}
   \end{figure} 
   
   For some subregion $[\lambda, \mu]$ of the region $[P,Q]$ and 
   some tiling $\tau$ of $[\lambda, \mu]$, 
   where $\lambda$ and $\mu$ are lattice paths from $(0,0)$ to $(m,r)$ in $[P,Q]$ 
   with $\lambda$ never going above $\mu$, 
   we define a \emph{block-tiled region} (abbreviated BTR) $[\lambda, \mu]_{\tau}$ as follows: 
      \begin{enumerate}
      \item
       Each maximal continuous intersection of $\lambda$ and $\mu$ passes 
       through an outside corner or an end point of $[P,Q]$.
      \item
      $\tau$ uses blocks as tiles. 
      That means the set of all the unit boxes in $[\lambda, \mu]$ is covered 
      by blocks without gaps or overlaps.
      \item
      If two blocks have a same-labeled-box in $[\lambda, \mu]_{\tau}$, 
      they are clones to each other.
      \end{enumerate}
   
   \noindent
   Note that a block-tiled region is not always defined for any subregion or any tilling.
   For two block-tiled regions $[\lambda, \mu]_{\tau}$ and $[\lambda', \mu']_{\tau'}$, 
   we say $[\lambda, \mu]_{\tau}$ covers $[\lambda', \mu']_{\tau'}$ 
   if $[\lambda, \mu]_{\tau}$ can be obtained by 
   attaching a clone of a block in $[\lambda', \mu']_{\tau'}$ below $\lambda'$ or above $\mu'$.
   If a block-tiled region is not covered by any other, it is called a \emph{maximal block-tiled region}.

   \begin{lemma}
   \label{faces-regions}
      There is a one-to-one correspondence 
      between the set of all the faces of $\mathcal{P}(M[P,Q])$
      and the set of all the maximal block-tiled regions inside $[P,Q]$.
   \end{lemma}   
   
   \begin{proof}
      For a face $\sigma$ of $\mathcal{P}(M[P,Q])$, take a set of matroid constructions by which $\sigma$ is obtained from $\mathcal{P}(M[P,Q])$.
      If $(p,q)$-direct sum is in the set where $(p,q)$ is an outside corner of $P$ or $Q$, 
      we remove all the steps inside $[P,Q]$ which lie strictly northwest or southeast of the point $(p,q)$ respectively.
      Also, if $i$-contraction or $i$-deletion is in the set, 
      all the East steps or North steps with the label $i$ inside $[P,Q]$ are removed respectively.
      After removing all the steps corresponding to constructions in the set, 
      if we delete remaining steps not on the connected lattice paths from $(0,0)$ to $(m,r)$ lastly, 
      we end up with a maximal block-tiled region $[\lambda, \mu]_{\tau}$ inside $[P,Q]$.
      Conversely, if some operations in $\mathcal{D}(P,Q) \cupdot \mathcal{C}(P,Q) \cupdot \mathcal{S}(P,Q)$, 
      which should be used to get the maximal block-tiled region $[\lambda, \mu]_{\tau}$ from $[P,Q]$, 
      are used to the polytope $\mathcal{P}(M[P,Q])$,
      one can easily check that the obtained corresponding face is the face $\sigma$.
   \end{proof}

   A \emph{block-tiled band} is a block-tiled region containing no $2 \times 2$ squares.
   We say block-tiled bands inside $[P,Q]$ are in the same \emph{family}
   if the sets of maximal continuous intersections and clones constituting them are the same.
   A \emph{block-tiled bottom} is a block-tiled band 
   which is the lowest one among block-tiled bands in its family.   
   Note that a block-tiled band bordered by $\lambda$ inside a maximal block-tiled region $[\lambda, \mu]_{\tau}$ 
   is a block-tiled bottom.
   We say that two blocks are \emph{adjacent} if they share a step on their boundaries. 
   Note that, if two adjacent blocks in a block-tiled region $[\lambda, \mu]_{\tau}$ share two or more steps, 
   they are clones each other.
   A block in a block-tiled band can be adjacent at most two other blocks. 
   
   There is a one-to-one correspondence 
   between the set of all the maximal block-tiled regions inside $[P, Q]$
   and the set of all the block-tiled bottoms inside $[P, Q]$.
   For a given maximal block-tiled region $[\lambda, \mu]_{\tau}$ inside $[P, Q]$, 
   one can remove all the clones except the lowest ones 
   to get a block-tiled bottom $[\lambda, \nu]_{\tau}$ inside $[P, Q]$
   where we identity $\lambda$ with the lower bounding path of 
   the Young diagram of $[\lambda, \mu] (= \lambda \setminus \mu)$, and  
   $\mu$ and $\nu$ with upper bounding paths of 
   the Young diagrams of $[\lambda, \mu]$ and $[\lambda, \nu]$ respectively.
   The inverse is obtained by inserting 
   all the clones of each block in the block-tiled bottom $[\lambda, \nu]_{\tau}$ 
   into a region $[\nu, \mu]$ 
   where $\mu$ is the upper bounding path of the highest block-tiled band among the family members of $[\lambda, \nu]_{\tau}$.
   See Example~\ref{eg-regions-bottoms}.
   Hence, there is a one-to-one correspondence 
   between the set of all the faces of $\mathcal{P}(M[P,Q])$ and the set of all the block-tiled bottoms inside $[P,Q]$
   by Lemma~\ref{faces-regions}.
   
   \begin{example}
      Let $P = E^3N^3EN^2$ and $Q = N^3ENE^2NE$.      
      The shaded block-tiled region shown in Figure~\ref{fig-regions-bottoms}(a)
      is a block-tiled bottom inside $[P,Q]$.            
      Since the block-tiled region in Figure~\ref{fig-regions-bottoms}(b) is obtained by 
      inserting the clone of the block labeled by $2$-$3$-$4$ in the shaded block-tiled bottom,
      it covers the block-tiled region in Figure~\ref{fig-regions-bottoms}(a).
      Hence, the block-tiled region in Figure~\ref{fig-regions-bottoms}(a) is not maximal.       
      If we also insert the clone of the single block labeled by $5$ as in Figure~\ref{fig-regions-bottoms}(c),
      the upper bound of the region in Figure~\ref{fig-regions-bottoms}(c) is not a lattice path.
      Hence, the region in Figure~\ref{fig-regions-bottoms}(c) is not a skew shape, and 
      not a block-tiled region.       
      Therefore, the striped block-tiled band in Figure~\ref{fig-regions-bottoms}(b) is 
      the highest family member of the shaded block-tiled bottom, and
      the blocked-tiled region in Figure~\ref{fig-regions-bottoms}(b) is
      the maximal block-tiled region corresponding to the shaded block-tiled bottom.     
   \label{eg-regions-bottoms}
   \end{example}
            
   \begin{figure}[h]
      \begin{center}
         \begin{tabular}{ccccc}
            \includegraphics[width = 0.15\textwidth]{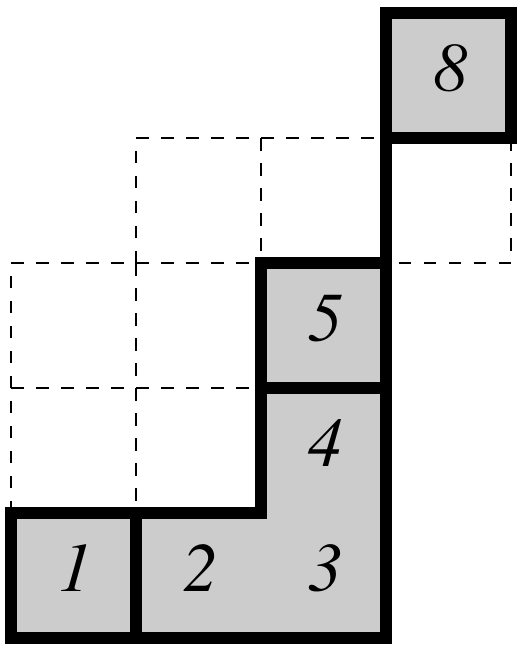} & &
            \includegraphics[width = 0.15\textwidth]{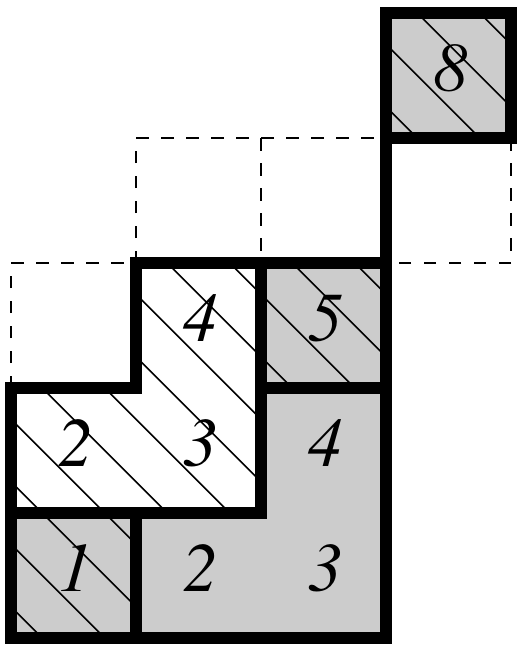} & &
            \includegraphics[width = 0.15\textwidth]{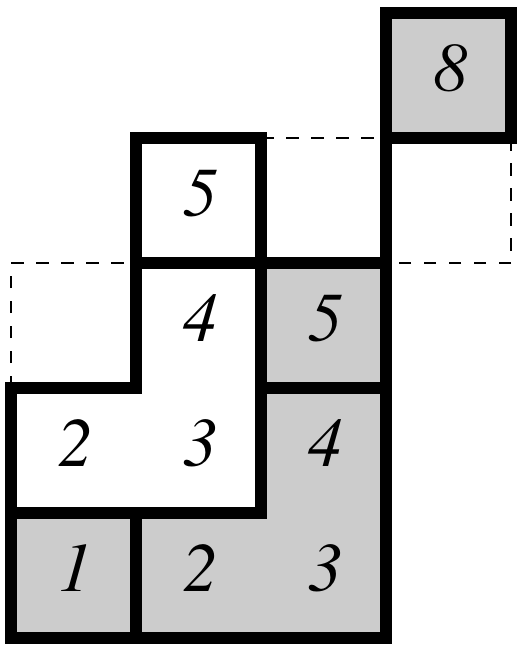} \\
            \footnotesize{(a) Not maximal} & &
            \footnotesize{(b) Maximal block-tiled region} & &
            \footnotesize{(c) Not a skew shape}
         \end{tabular}
      \end{center}  
   \caption{Correspondence between block-tiled bottoms and maximal block-tiled regions}
   \label{fig-regions-bottoms}
   \end{figure}
      
   The next proposition describes the covering relation in the face poset of $\mathcal{P}(M[P,Q])$ 
   in terms of block-tiled bottoms. 
   The proof is straightforward and is omitted.
   
   \begin{proposition}
   \label{covering-relation}
      Codimension $1$ subfaces of an $n$-dimensional face corresponding to 
      a block-tiled bottom $[\lambda, \nu]_\tau$ inside $[P,Q]$ are obtained as follows:
      
      \begin{enumerate}
      \item[(1)]
      (Direct sum at an outside corner)
      For an outside corner $(p,q)$ of $[P,Q]$,
      let $[P,Q]_{(p,q)}$ be the subregion of $[P,Q]$ 
      which can be obtained after $(p,q)$-direct sum works to $[P,Q]$. 
      If there exists a family member $f$ of $[\lambda, \nu]_\tau$ such that 
      $f \setminus [P,Q]_{(p,q)}$ consists of a single block containing 
      a starting point $(p,*)$, an ending point $(*,q)$, and an outside corner $(p,q)$,
      then one can take the lowest one among such $f$'s and 
      remove the single block keeping the lattice path through $(p,*)$, $(p,q)$, and $(*,q)$.
      See Example~\ref{eg-direct-sum-upper}.

      \item[(2)]
      (Deletion of a block) 
      Let $i$ be the smallest box label of a block $B_1$ in $[\lambda, \nu]_\tau$.
      If $B_1$ is adjacent to only one block $B_2$ whose labels are bigger than those of $B_1$, 
      delete $B_1$ from the family members of $[\lambda, \nu]_\tau$
      keeping the perimeter of $B_1$ from $(0,0)$ to the starting point of the clone of $B_2$
      so that new block-tiled bands with one less blocks than $[\lambda, \nu]_\tau$ are obtained.
      Note that the new bands fall into at most two types of families.
      One can get the block-tiled bottoms corresponding to $(n-1)$-dimensional subfaces
      by taking the lowest ones by families.         
      If the perimeter begins with an East step (a North step), 
      the obtained bottom corresponds to the $i$-deletion (respectively, $i$-contraction).
      See Figure~\ref{fig-deletion-contraction}(c) and~\ref{fig-deletion-contraction}(d) 
      in Example~\ref{eg-deletion-contraction}.
      If $B_1$ is adjacent to only one block whose labels are smaller than those of $B_1$,
      similar operations can be done.
            
      If $B_1$ is not adjacent to any other block, 
      one can either replace $\lambda$ by removing all the boxes of $B_1$,
      or replace $\nu$ by adding all the boxes of $B_1$.
      This corresponds to the $i$-contraction or the $i$-deletion, respectively.
      See Figure~\ref{fig-deletion-contraction}(e) and~\ref{fig-deletion-contraction}(f) 
      in Example~\ref{eg-deletion-contraction}.
      \item[(3)]
      (Merge of adjacent blocks)
      If two blocks $B_1$ and $B_2$ in $[\lambda, \nu]_\tau$ 
      are adjacent along a step in this order, 
      one can merge $B_1$ and the clone of~$B_2$ in the family members of $[\lambda, \nu]_\tau$
      by deleting the step between $B_1$ and the clone of $B_2$
      so that new block-tiled bands with one less blocks than $[\lambda, \nu]_\tau$ are gained.
      Note that the new bands also fall into at most two types of families.
      One can obtain the block-tiled bottoms corresponding to $(n-1)$-dimensional subfaces
      by taking the lowest ones by families.
      If the deleted step is a North step (an East step), 
      this corresponds to the $i$-deletion (respectively, $i$-contraction). 
      See Figure~\ref{fig-deletion-contraction}(g) and~\ref{fig-deletion-contraction}(h) 
      in Example~\ref{eg-deletion-contraction}.
      \end{enumerate}
   \label{prop-cover-relation}   
   \end{proposition}

   \begin{example}
      For the region $[P,Q]$ where $P = E^5N^2E^2NE^2N^4$ and $Q = N^4E^4N^3E^5$, 
      one can have an outside corner $(4,4)$ of $Q$ and a block-tiled bottom 
      as in Figure~\ref{fig-direct-sum-upper}(a).
      Note that the maximal block-tiled region corresponding to the block-tiled bottom 
      in Figure~\ref{fig-direct-sum-upper}(b)
      consists of all the family members of the given bottom. 
      The shaded block-tiled band shown in Figure~\ref{fig-direct-sum-upper}(c)
      is the lowest family member satisfying conditions 
      in case (1) of Proposition~\ref{prop-cover-relation}. 
      One can obtain the new block-tiled bottom from the lowest family member
      after removing the $8$-labeled block 
      in Figure~\ref{fig-direct-sum-upper}(d).
   \label{eg-direct-sum-upper}
   \end{example}
   
   \begin{figure}[h] 
      \begin{center}
         \begin{tabular}{ccc}
            \includegraphics[width = 0.3\textwidth]{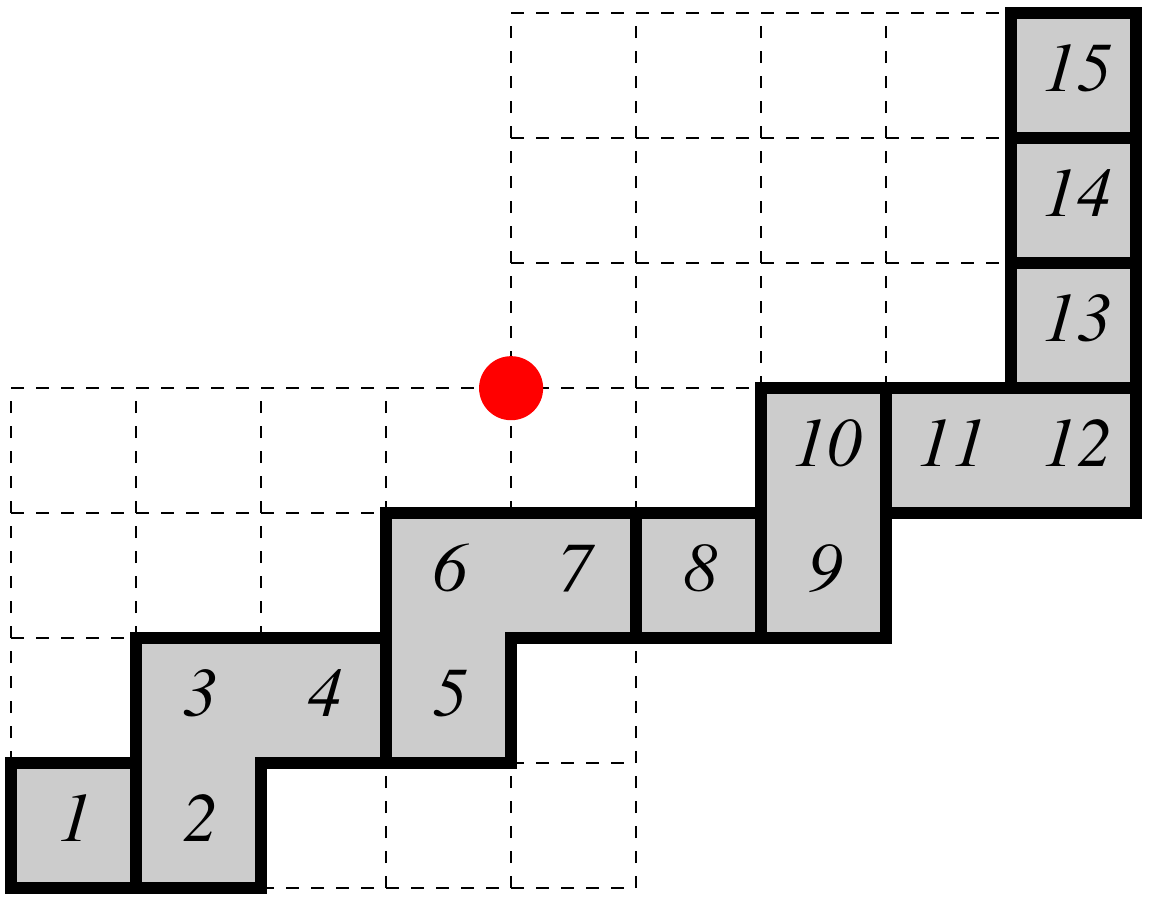}& \hspace{1cm} &
            \includegraphics[width = 0.3\textwidth]{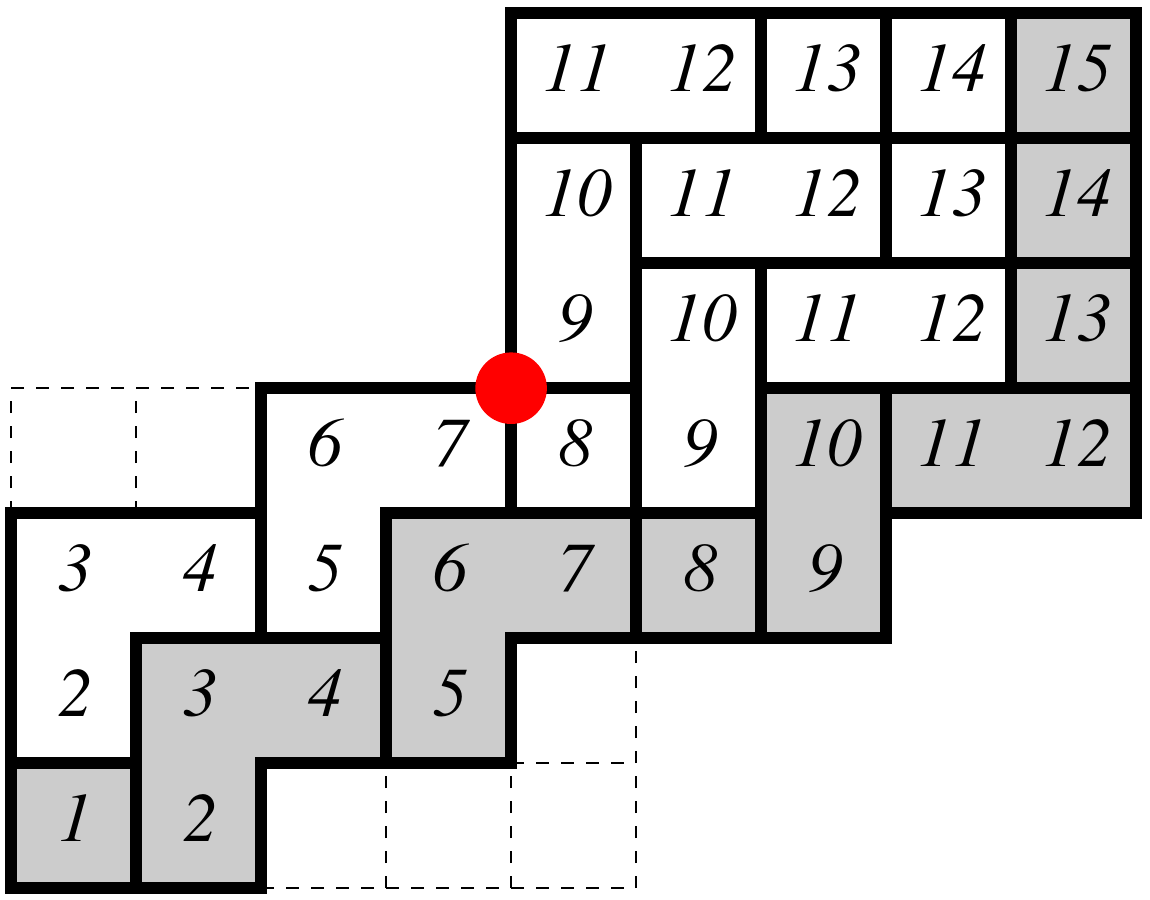}\\
            \footnotesize{(a) Block-tiled bottom} & &
            \footnotesize{(b) Maximal block-tiled region}\\\\
            \includegraphics[width = 0.3\textwidth]{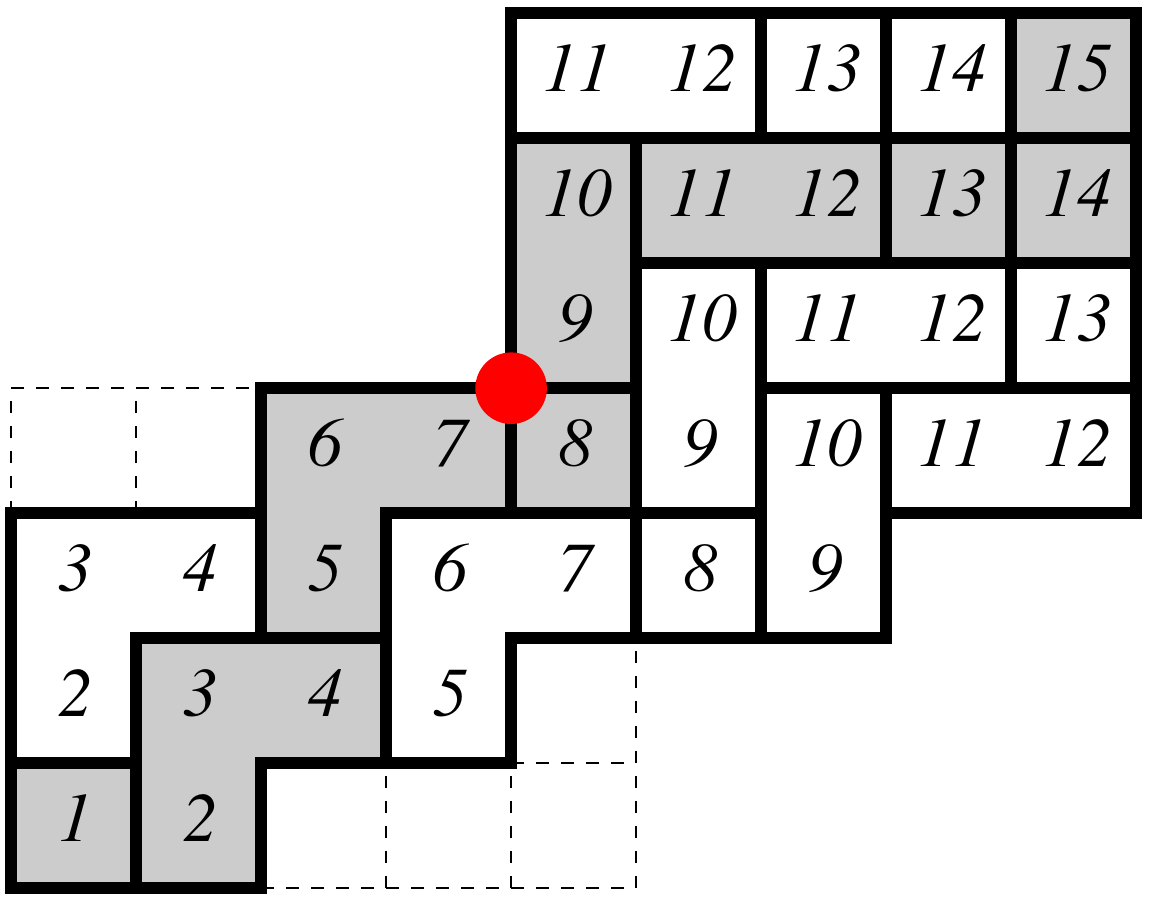}& \hspace{1cm} &
            \includegraphics[width = 0.3\textwidth]{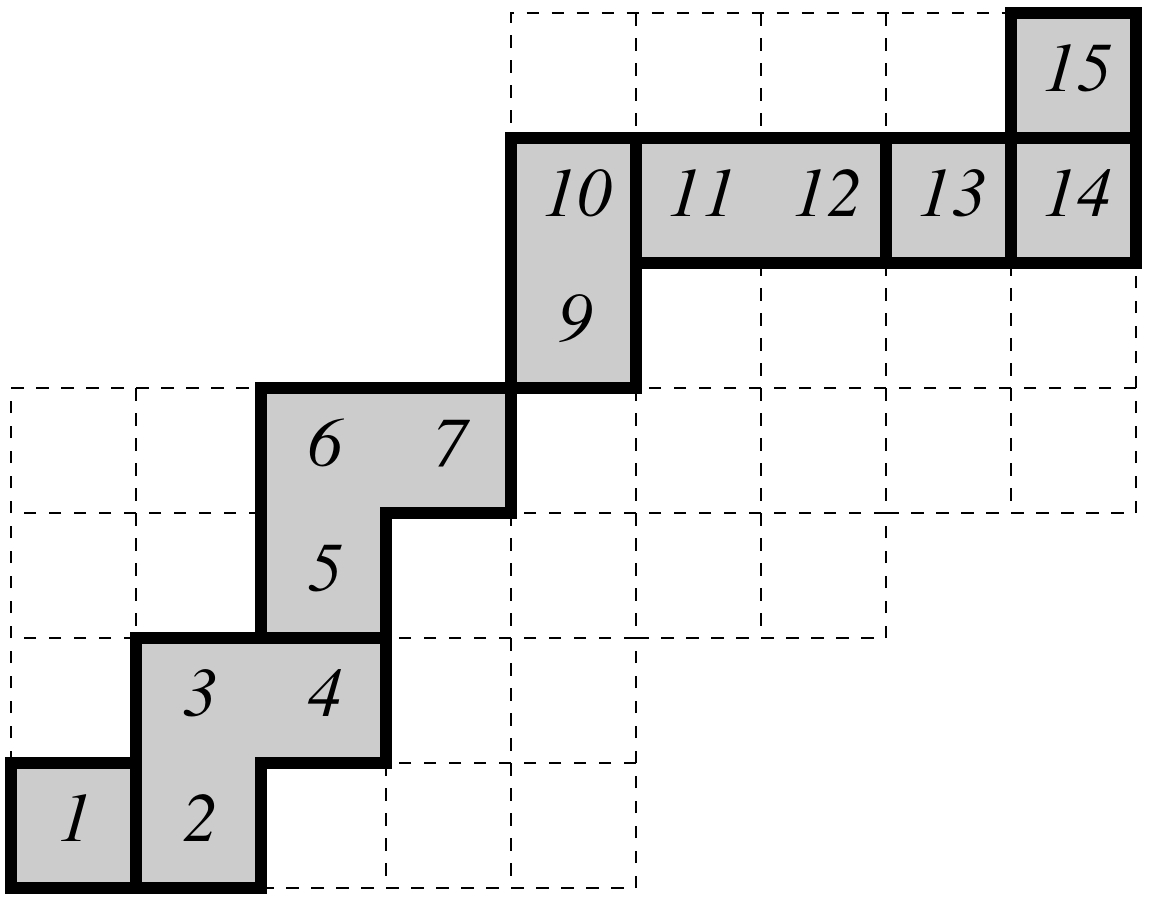}\\
            \footnotesize{(c) Lowest family member}& &
            \hspace{1cm}\footnotesize{(d) Block-tiled bottom} \qquad \quad \quad\\
            \hspace{0.7cm}\footnotesize{as in Proposition~\ref{prop-cover-relation}(1)}& &
            \hspace{1cm}\footnotesize{after $(4,4)$-direct sum}
         \end{tabular}
      \end{center}
   \caption{Block-tiled bottoms corresponding to faces of $\mathcal{P}(M[P,Q])$}
   \label{fig-direct-sum-upper}
   \end{figure}
      
   \begin{example}
      Let $P = E^2NE^3N^4EN$ and $Q = N^3EN^2E^2NE^3$.
      Figure~\ref{fig-deletion-contraction}(a) shows 
      the maximal block-tiled region for a $7$-dimensional face of $\mathcal{P}(M[P,Q])$.
      The maximal block-tiled region for the face obtained from $(5,5)$-direct sum from this face
      is shown in Figure~\ref{fig-deletion-contraction}(b).
      Figures~\ref{fig-deletion-contraction}(c)-(h) show 
      codimension $1$ faces of the $6$-dimensional face corresponding to 
      the maximal block-tiled region shown in Figure~\ref{fig-deletion-contraction}(b).
      They are examples for cases (2) and (3) of Proposition~\ref{covering-relation}.
      \label{eg-deletion-contraction}
   \end{example}

   \begin{figure}[h]
      \begin{center}
         \begin{tabular}{ccccccc}
            \includegraphics[width = 0.13\textwidth]{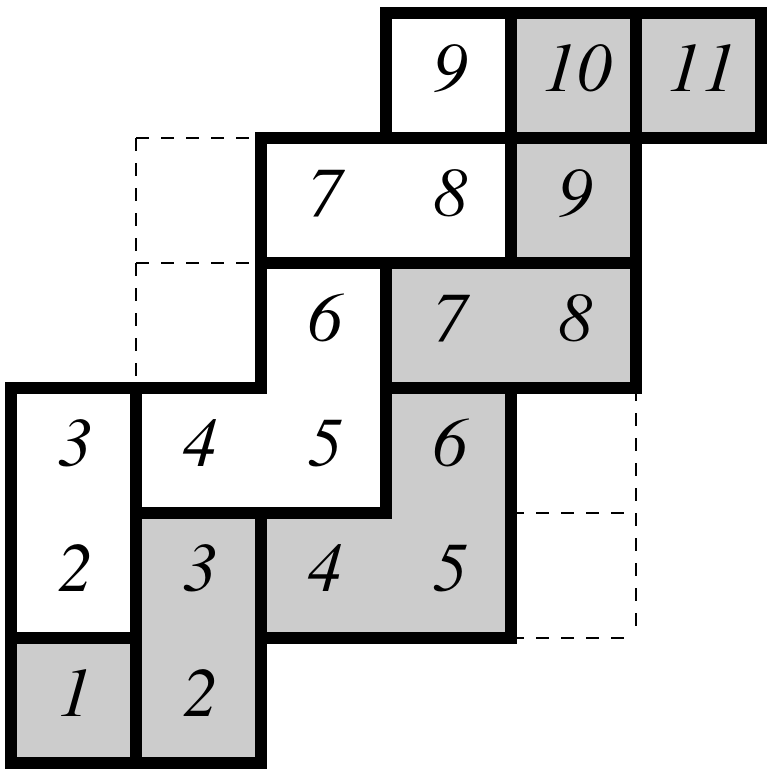}&  \hspace{0.05cm} &
            \includegraphics[width = 0.13\textwidth]{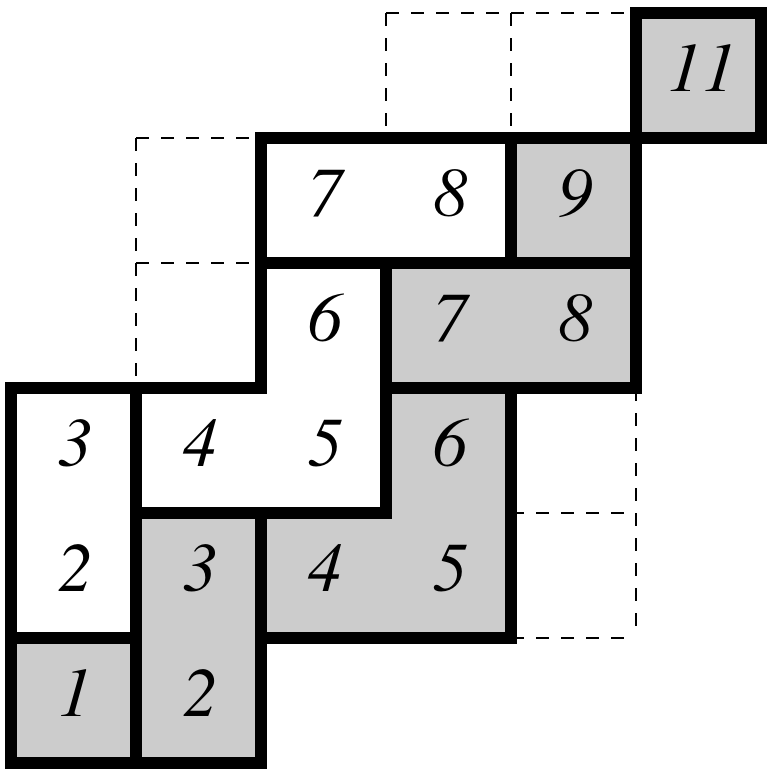}&  \hspace{0.05cm} &
            \includegraphics[width = 0.13\textwidth]{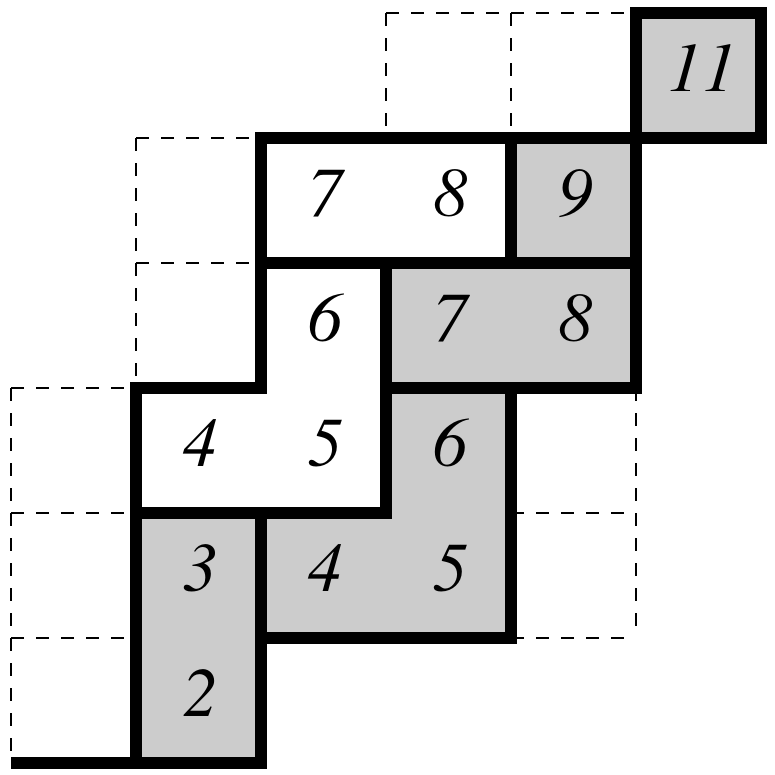}& \hspace{0.05cm}  &
            \includegraphics[width = 0.13\textwidth]{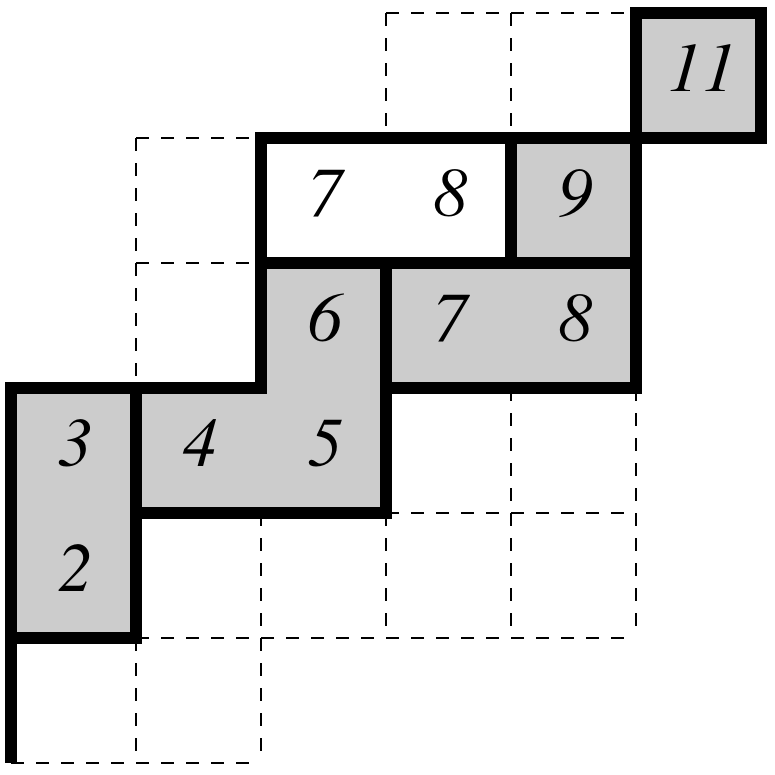}\\
            \footnotesize{(a) A maximal BTR} & & 
            \footnotesize{(b) $(5,5)$-direct sum} & & 
            \footnotesize{(c) $1$-deletion} & & 
            \footnotesize{(d) $1$-contraction}\\
            \includegraphics[width = 0.13\textwidth]{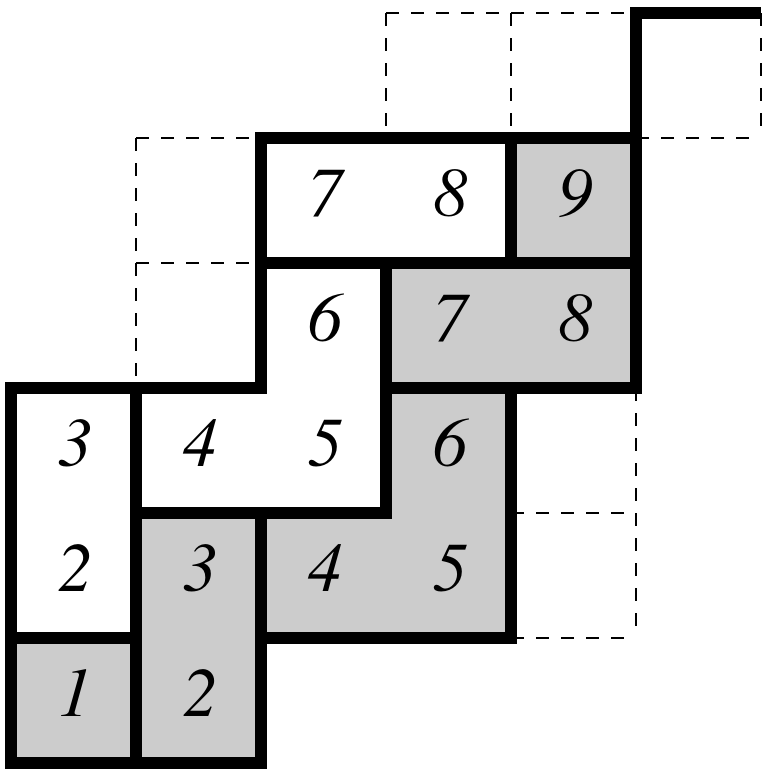}&  &
            \includegraphics[width = 0.13\textwidth]{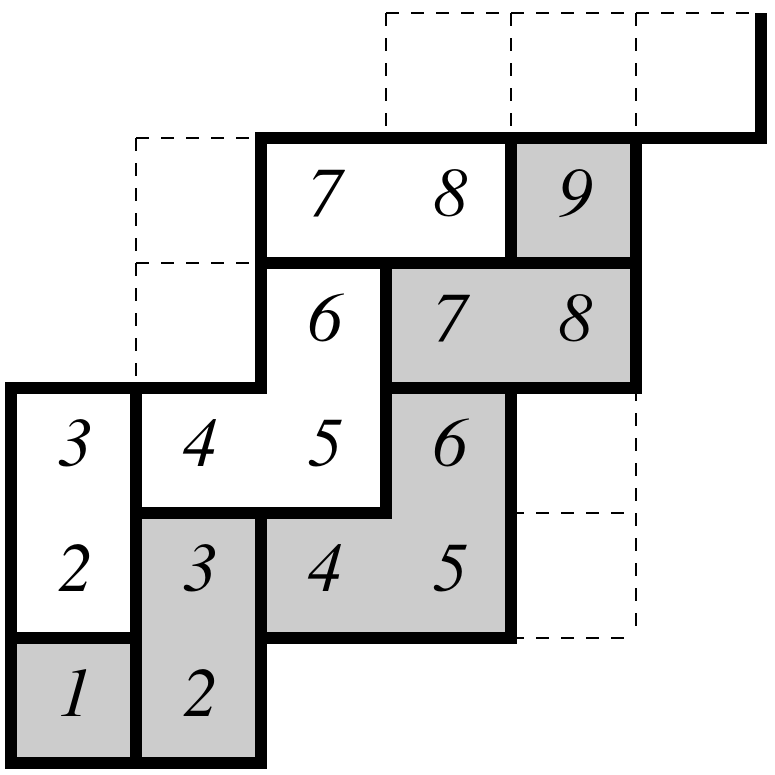}&  &
            \includegraphics[width = 0.13\textwidth]{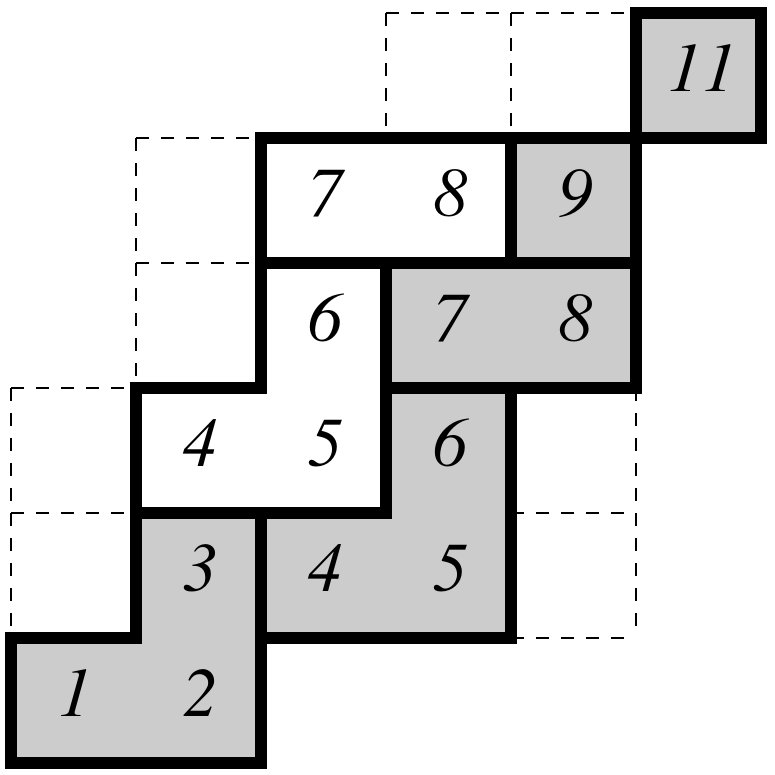}&  &
            \includegraphics[width = 0.13\textwidth]{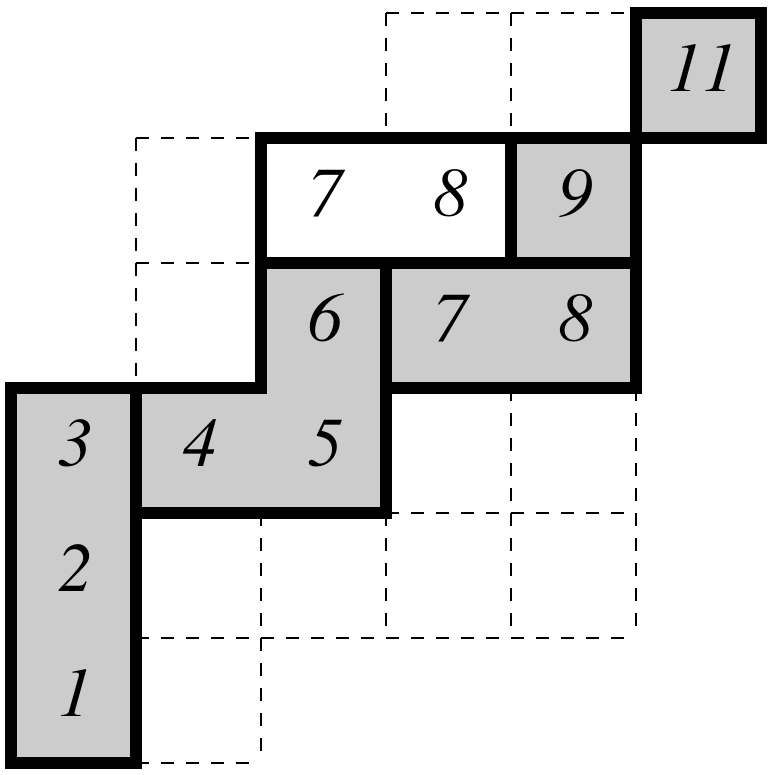}\\
            \footnotesize{(e) $11$-contraction} & & 
            \footnotesize{(f) $11$-deletion} & & 
            \footnotesize{(g) $2$-deletion} & & 
            \footnotesize{(h) $2$-contraction}\\
         \end{tabular}
      \end{center}
   \caption{Direct sum/Deletion/Contraction}
   \label{fig-deletion-contraction}
   \end{figure}
   
   The following corollary follows from Proposition~\ref{covering-relation}.
   
   \begin{corollary}
   \label{faces-using-border-strips}
      All the $n$-dimensional faces of the polytope $\mathcal{P}(M[P,Q])$ are corresponding to 
      the block-tiled bottoms with $n$ blocks inside the region $[P,Q]$.
   \end{corollary}
   
   \begin{proof}
      By Proposition~\ref{dimension-of-matroid-polytope}, 
      the dimension of a lattice path matroid polytope $\mathcal{P}(M[P,Q])$ having $m+r-1$ single-boxed-blocks 
      in the block-tiled bottom of $[P,Q]$ is $m+r-1$.
      The result follows since 
      each covering relation in the face poset of $\mathcal{P}(M[P,Q])$ reduces 
      the number of blocks in the block-tiled bottom by one. 
   \end{proof}
   
   In the following proposition, we give a more explicit proof for a special case of Corollary~\ref{faces-using-border-strips}.
  
   \begin{proposition}
   All the edges of the polytope $\mathcal{P}(M[P,Q])$ are corresponding to 
   the block-tiled bottoms with $1$ block inside the region $[P,Q]$. 
   \label{1-block}
   \end{proposition}
  
   \begin{proof}
   Take an edge $e_{B} e_{B'}$ of the polytope $\mathcal{P}(M[P,Q])$ where
   \begin{align*}
      e_{B} = e_{b_1} + \cdots + e_{b_r} &= (x_1, x_2, \ldots, x_n) \in \DD_n ~\text{and}\\   
      e_{B'} = e_{b'_1} + \cdots + e_{b'_r} &= (y_1, y_2, \ldots, y_n) \in \DD_n 
   \end{align*}
   for  $(n-1)$-simplex $\DD_n = \{ (x_1, \dots, x_n) \in \reals^n : x_1 \ge 0, \dots, x_n \ge 0, x_1 + \cdots + x_n = r \}$.
   By Theorem~\ref{thm-matroid-polytopes}, without loss of generality, 
   there exist  $j$ and $k$ ($1 \leq j < k \leq n$) such that
   $$\left( \begin{matrix} x_j&x_k\\ y_j&y_k \end{matrix} \right) = \left( \begin{matrix} 1&0\\ 0&1 \end{matrix} \right),$$ 
   and $x_i = y_i$ ($1 \leq i \leq n$) for the other coordinates.   
   
   Because the associated lattice path $P(B)$ is identical with $P(B')$ before the $j$th step 
   and $\binom{x_j}{y_j} = \binom{1}{0}$, 
   we have a starting point of a region $R_1$ at the $j$th step. 
   Similarly, we have an ending point of a region $R_2$ at the $k$th step. 
   Since $P(B)$ and $P(B')$ also have identical sequences between $j$th step and $k$th step 
   after being separated only by $1$ step at $j$th step, 
   $P(B)$ and $P(B')$ do not intersect between $j$th step and $k$th step, and 
   $2 \times 2$ squares cannot be contained between $P(B)$ and $P(B')$. 
   Hence, $R_1$ and $R_2$ are the same skew-shaped region without $2 \times 2$ squares. 
   Therefore, there is a unique block generated by $P(B)$ and $P(B')$ inside $[P,Q]$.
   \end{proof}

   Consider the lattice path matroid polytope $\mathcal{P}(M[P,Q])$ and a lattice path~$L$ inside the region $[P,Q]$.
   We define $L'= L'(P,L)$ as the lattice path such that
   $L'$ passes all the intersection points of $P$ and $L$ and,
   for each maximal connected region between $P$ and $L$ with the starting point $(x,y)$ and the ending point $(x+a,y+b)$,
   $L'$ has a sequence $E^aN^b$ from $(x,y)$ to $(x+a,y+b)$.   
   
   \begin{corollary} 
   For the lattice path matroid polytope $\mathcal{P}(M[P,Q])$,
   the number of edges of this polytope is equal to the sum of the areas between $L$ and $L'$
   where the sum is over all lattice paths $L$ from $(0,0)$ to $(m,r)$ inside the region $[P,Q]$.
   \label{cor-edge-area}
   \end{corollary}
  
   \begin{proof}
   For a lattice path $L$ from $(0,0)$ to $(m,r)$ inside the region $[P,Q]$
   and a unit box $U$ with the starting point $(u_1,u_2)$ inside the region $[L',L]$,
   one can construct the block-tiled bottom $[\lambda, L]_\tau$ inside $[P,Q]$ 
   such that $[\lambda, L]_\tau$ has only one block with the starting point $(*,u_2)$ and the ending point $(u_1+1,*)$ on $L$,
   and $\lambda$ is identical with $L$ before $(*,u_2)$ and after $(u_1+1,*)$.   
   If a block-tiled bottom $[\lambda, \nu]_\tau$ with $1$ block inside $[P,Q]$ is given,
   one can take a lattice path $\nu$ and a unit box with the starting point $(u_1-1, u_2)$
   where $u_1$ is the $x$-coordinate of the ending point and $u_2$ is the $y$-coordinate of the starting point of the block. 
   
   Hence, there is a one to one correspondence 
   between block-tiled bottoms with $1$ block inside $[P,Q]$ and 
   the pairs of a lattice path $L$ from $(0,0)$ to $(m,r)$ inside $[P,Q]$ and a unit box inside $[L',L]$.
   By Proposition~\ref{1-block}, 
   the number of edges of the polytope $\mathcal{P}(M[P,Q])$ is equal to the sum of areas between $L$ and $L'$
   with summation over all lattice paths $L$ from $(0,0)$ to $(m,r)$ inside the region $[P,Q]$.
   \end{proof}     

   Corollary~\ref{cor-edge-area} is a nice generalization of the following result~\cite[Lemma 3.6]{Bidkhori}.
   Note that $L'=P$ in the case $P$ is the lattice path $E^mN^r$.
         
   \begin{corollary} 
   \label{1-box}
   Consider the lattice path matroid polytope $\mathcal{P}(M[E^mN^r,Q])$ 
   where $Q$ is a lattice path from $(0,0)$ to $(m,r)$. 
   The number of edges of this polytope is equal to the sum of areas below the path $L$
   with summation over all lattice paths $L$ from $(0,0)$ to $(m,r)$ inside the region $[E^mN^r,Q]$.
   \end{corollary}

\section{Future works}
\label{sec-future-works}

   The $\C\D$-index $\Psi (\mathcal{P})$ of a polytope $\mathcal{P}$, 
   a polynomial in the noncommutative variables $\C$ and $\D$, 
   is a very compact encoding of the flag numbers of a polytope $\mathcal{P}$~\cite{BayerKlapper}.
   The third author shows how the $\C\D$-index of a polytope can be expressed 
   when a polytope is cut by a hyperplane~\cite{SangwookFlag}.
   For lattice path matroids of rank $2$, the following is obtained from~\cite{SangwookFlag}.
   
   \begin{proposition}
      For lattice paths $P=E^{\alpha+\beta} N E^{\gamma} N$ and 
      $Q=NE^{\alpha} NE^{\beta+\gamma}$
      with $\alpha+\beta+\gamma = m$,
      the $\C\D$-index of the lattice path matroid polytope $\mathcal{P}(M[P,Q])$ is 
      {\footnotesize
      \begin{multline*}
      \Psi(\mathcal{P}(M[P,Q])) 
      = \sum_{i = \alpha+1}^{\alpha+\beta} \Psi(\mathcal{P}(M_i))
      - \left( \sum_{i=\alpha+2}^{\alpha+\beta} \Psi(\DD_i \times \DD_{m-i+2}) \right) \cdot \C \\
      - \hspace{-3mm} \sum_{(0,0) < (i,j) \le (\alpha, \gamma)} 
                           \hspace{-2mm} 
                          \binom{\alpha+1}{i} \binom{\gamma+1}{j}
                           \hspace{-1mm} 
                            \left( \sum_{k=2}^{\beta-2} 
                                   \Psi(\DD_{\alpha-i+k} \times \DD_{\beta-j+\gamma-k+2}) \right)
                            \cdot \D \cdot \Psi(\DD_{i+j}),
      \end{multline*}
      }
      where $M_i = M[E^i N E^{m-i} N, N E^{i-1} N E^{m-i+1}]$ and 
      $\DD_j$ is $(j-1)$-dimensional simplex.
   \end{proposition}

   It is known that a hyperplane split of a lattice path matroid polytope $\mathcal{P}(M[P,Q])$
   can occur when $[P,Q]$ contains a $2 \times 2$ square~\cite{Bidkhori}. 
   Using descriptions of faces of a lattice path matroid polytope appeared in 
   Sections~\ref{sec-border strips} and \ref{sec-skew-shape},
   we would like to understand the $\C\D$-index of a lattice path matroid polytope of rank greater than $2$.

\bibliographystyle{plain}

\begin{thebibliography}{10}

\bibitem{ArdilaRinconWilliams}
Federico Ardila, Felipe Rinc{\'o}n, and Lauren Williams.
\newblock Positroids and non-crossing partitions.
\newblock {\em Trans. Amer. Math. Soc.}, 368(1):337--363, 2016.

\bibitem{BayerKlapper}
Margaret~M. Bayer and Andrew Klapper.
\newblock A new index for polytopes.
\newblock {\em Discrete Comput. Geom.}, 6(1):33--47, 1991.

\bibitem{Bidkhori}
Hoda Bidkhori.
\newblock Lattice path matroid polytopes.
\newblock {\em {\tt arXiv:1212.5705}}, 2012.

\bibitem{BoninMierNoy}
Joseph Bonin, Anna de~Mier, and Marc Noy.
\newblock Lattice path matroids: enumerative aspects and {T}utte polynomials.
\newblock {\em J. Combin. Theory Ser. A}, 104(1):63--94, 2003.

\bibitem{BoninMier}
Joseph~E. Bonin and Anna de~Mier.
\newblock Lattice path matroids: structural properties.
\newblock {\em European J. Combin.}, 27(5):701--738, 2006.

\bibitem{DelucchiDlugosch}
Emanuele Delucchi and Martin Dlugosch.
\newblock Bergman complexes of lattice path matroids.
\newblock {\em SIAM J. Discrete Math.}, 29(4):1916--1930, 2015.

\bibitem{Edmonds}
Jack Edmonds.
\newblock Submodular functions, matroids, and certain polyhedra.
\newblock In {\em Combinatorial optimization---Eureka, you shrink!}, volume
  2570 of {\em Lecture Notes in Comput. Sci.}, pages 11--26. Springer, Berlin,
  2003.

\bibitem{EdmondsFulkerson}
Jack Edmonds and D.~R. Fulkerson.
\newblock Transversals and matroid partition.
\newblock {\em J. Res. Nat. Bur. Standards Sect. B}, 69B:147--153, 1965.

\bibitem{FeichtnerSturmfels}
Eva~Maria Feichtner and Bernd Sturmfels.
\newblock Matroid polytopes, nested sets and {B}ergman fans.
\newblock {\em Port. Math. (N.S.)}, 62(4):437--468, 2005.

\bibitem{GelfandGoreskyMacPhersonSerganova}
I.~M. Gel{\cprime}fand, R.~M. Goresky, R.~D. MacPherson, and V.~V. Serganova.
\newblock Combinatorial geometries, convex polyhedra, and {S}chubert cells.
\newblock {\em Adv. in Math.}, 63(3):301--316, 1987.

\bibitem{SangwookFlag}
Sangwook Kim.
\newblock Flag enumerations of matroid base polytopes.
\newblock {\em J. Combin. Theory Ser. A}, 117(7):928--942, 2010.

\bibitem{MortonTurner}
Jason Morton and Jacob Turner.
\newblock Computing the {T}utte polynomial of lattice path matroids using
  determinantal circuits.
\newblock {\em Theoret. Comput. Sci.}, 598:150--156, 2015.

\bibitem{Oh}
Suho Oh.
\newblock Generalized permutohedra, {$h$}-vectors of cotransversal matroids and
  pure {O}-sequences.
\newblock {\em Electron. J. Combin.}, 20(3):Paper 14, 14, 2013.

\bibitem{PostnikovReinerWilliams}
Alex Postnikov, Victor Reiner, and Lauren Williams.
\newblock Faces of generalized permutohedra.
\newblock {\em Doc. Math.}, 13:207--273, 2008.

\bibitem{Postnikov}
Alexander Postnikov.
\newblock Permutohedra, associahedra, and beyond.
\newblock {\em Int. Math. Res. Not. IMRN}, (6):1026--1106, 2009.

\bibitem{Schweig2010}
Jay Schweig.
\newblock On the {$h$}-vector of a lattice path matroid.
\newblock {\em Electron. J. Combin.}, 17(1):Note 3, 6, 2010.

\end{thebibliography}

\def\cprime{$'$}

\end{document}